\documentclass[a4paper,reqno,11pt]{amsart}
\usepackage{amsfonts,amsmath,amsthm,amssymb,stmaryrd}
\newtheorem{theorem}{Theorem}[section]
\newtheorem{lemma}[theorem]{Lemma}

\newtheorem{Remark}[theorem]{Remark}

\numberwithin{equation}{section}
\allowdisplaybreaks

\usepackage[left=1 in, right=1 in,top=1 in, bottom=1 in]{geometry}

\DeclareMathOperator{\diverge}{div} 
 \providecommand{\norm}[1]{\left\Vert#1\right\Vert}

\def\r3{\mathbb{R}^3}

\begin{document}
\title[Decay of dissipative equations]{Decay of dissipative equations and
negative Sobolev spaces}
\author{Yan Guo}
\address{Division of Applied Mathemathics\\
Brown University\\
Providence, RI 02912, USA}
\email[Y. Guo]{guoy@cfm.brown.edu}
\author{Yanjin Wang}
\address{School of Mathematical Sciences\\
Xiamen University\\
Xiamen, Fujian 361005, China}
\email[Y. J. Wang]{yanjin$\_$wang@xmu.edu.cn}
\thanks{Corresponding author: Y. J. Wang, yanjin$\_$wang@xmu.edu.cn. Y. Guo's research is supported in part by NSF \#0905255 as well as
a Chinese NSF grant \#10828103. Y. J. Wang's research is supported in part by
National Natural Science Foundation of China-NSAF \#10976026.}


\begin{abstract}
We develop a general energy method for proving the optimal time decay rates
of the solutions to the dissipative equations in the whole space. Our method
is applied to classical examples such as the heat equation, the compressible
Navier-Stokes equations and the Boltzmann equation. In particular, the
optimal decay rates of the higher-order spatial derivatives of solutions are
obtained. The negative Sobolev norms are shown to be preserved along time
evolution and enhance the decay rates. We use a family of scaled energy
estimates with minimum derivative counts and interpolations among them
without linear decay analysis.
\end{abstract}

\maketitle



\section{Introduction}


The purpose of this paper is to develop a new method to establish the optimal
time decay rates of the solutions to the Cauchy problem for the compressible
Navier-Stokes equations and the Boltzmann equation through the pure energy
method. Both the two equations can be formulated as the perturbed operator
form:
\begin{equation}
\left\{
\begin{array}{lll}
\partial _{t}U+\mathfrak{L}U=\mathfrak{N}(U) &  &  \\
U|_{t=0}=U_{0}, &  &
\end{array}%
\right.   \label{dissipative equation}
\end{equation}%
where $U$ is the small perturbation of the equilibrium state, $\mathfrak{L}$
is the linear operator and $\mathfrak{N}(U)$ is the nonlinear term. The
linear operator $\mathfrak{L}$ is positively definite in some sense, which
implies that the solution $e^{t\mathfrak{L}}U_{0}$ of the linearized
equation of \eqref{dissipative equation} converges to $0$ as $t\rightarrow
\infty $. By the classical spectral method, the optimal time decay rates of
the linearized equations of the compressible Navier-Stokes equations and the
Boltzmann equation are well known. The decay rates are similar to that of
the heat equation: for $1\leq p\leq 2$,
\begin{equation}
\norm{\nabla_x^\ell e^{t \mathfrak{L}}U_0}_{L^{2}}\leq C(1+t)^{-\frac{3}{2}(%
\frac{1}{p}-\frac{1}{2})-\frac{\ell }{2}}\left( \norm{U_0}_{L^{p}}+%
\norm{\nabla_x^\ell U_0}_{L^{2}}\right) ,\ \ell \geq 0.  \label{decay rate}
\end{equation}%
One may then expect that the small solution of the nonlinear equation %
\eqref{dissipative equation} has the same decay rate as the linear one %
\eqref{decay rate}. Many works were denoted to proving the time decay rate
for the nonlinear system \eqref{dissipative equation}. For instance, see
\cite{D1,D2,DUYZ1,HZ1,HZ2,KK1,KK2,K1,KS1,LW,M,MN2,P} for the compressible
Navier-Stokes equations and \cite{DS1,DS2,DUYZ2,NI1976,UY2006,UYZ,YY,YZ} for
the Boltzmann equation, and the references therein. There are two main kinds
of method for proving these decay rates among those references. One is that
under the additional assumption that $U_{0}\in L^{p}$ with $1\leq p<2$ (near
$1$), then the optimal decay rate of \eqref{dissipative equation} is proved
by combining the linear optimal decay rate \eqref{decay rate} of spectral
analysis and the energy method, cf. \cite%
{DS1,DS2,DUYZ1,DUYZ2,HZ1,HZ2,KK1,KK2,K1,KS1,LW,MN2,NI1976,P,UY2006,YY}. The
other one is to proving the decay rate through the pure energy method, cf.
\cite{D1,D2,M,UYZ,YZ}.

It is difficult to show that the $L^{p}$ norm of the solution can be
preserved along time evolution in the $L^{p}$--$L^{2}$ approach. On the other
hand, except \cite{M}, the existing pure energy method of proving the decay
rate does not lead to the optimal decay rate for the solution.  Motivated by
\cite{GT},  using a negative Sobolev space $\dot{H}^{-s}$ to replace $L^{p}$
norm, we combine scaled energy estimates with the interpolation between
negative and positive Sobolev norms to prove the time decay rate for these
dissipative equations. To illustrate the main idea of our approach,
\smallskip \smallskip we first revisit the heat equation
\begin{equation}
\left\{
\begin{array}{lll}
\partial _{t}u-\Delta u=0\ \text{ in }\mathbb{R}^{3} &  &  \\
u|_{t=0}=u_{0}, &  &
\end{array}%
\right.   \label{heat equation}
\end{equation}

\noindent \textbf{Notation 1.} In this paper, $\nabla ^{\ell }$ with an
integer $\ell \geq 0$ stands for the usual any spatial derivatives of order $%
\ell $. When $\ell <0$ or $\ell $ is not a positive integer, $\nabla ^{\ell }
$ stands for $\Lambda ^{\ell }$ defined by \eqref{1Lambdas}. We use $\dot{H}%
^{s}(\mathbb{R}^{3}),s\in \mathbb{R}$ to denote the homogeneous Sobolev
spaces on $\mathbb{R}^{3}$ with norm $\norm{\cdot}_{\dot{H}^{s}}$ defined by %
\eqref{1snorm}, and we use $H^{s}(\mathbb{R}^{3})$ to denote the usual
Sobolev spaces with norm $\norm{\cdot}_{H^{s}}$ and $L^{p},1\leq p\leq
\infty $ to denote the usual $L^{p}(\mathbb{R}^{3})$ spaces with norm $%
\norm{\cdot}_{L^{p}}$. We will employ the notation $a\lesssim b$ to mean
that $a\leq Cb$ for a universal constant $C>0$ that only depends on the parameters coming from the problem, and the indexes $N$ and $s$ coming from the regularity on the data. We also use $C_0$ for a positive constant depending additionally on the initial data.

\begin{theorem}
\label{heat} If $u_0\in H^N(\mathbb{R}^3)\cap \dot{H}^{-s}(\mathbb{R}^3)$
with $N\ge 0$ be an integer and $s\ge0$ be a real number, then for any real
number $\ell\in[-s, N]$, we have
\begin{equation}  \label{heat decay}
\norm{\nabla^\ell u(t)}_{L^2}\le C_0(1+t)^{-\frac{\ell+s}{2}}.
\end{equation}
\end{theorem}

\begin{proof}
Let $-s\leq \ell \leq N$. First, we have the standard energy identity of %
\eqref{heat equation}:
\begin{equation}
\frac{1}{2}\frac{d}{dt}\norm{\nabla^\ell u}_{L^{2}}^{2}+\norm{\nabla^{%
\ell+1} u}_{L^{2}}^{2}=0.  \label{heat 1}
\end{equation}%
Integrating the above in time, we obtain
\begin{equation}
\norm{\nabla^\ell u(t)}_{L^{2}}^{2}\leq \norm{\nabla^\ell u_0}_{L^{2}}^{2}.
\label{heat 2}
\end{equation}%
This gives in particular \eqref{heat decay} with $\ell =-s$. Now for $%
-s<\ell \leq N$, \ by Lemma \ref{1-sinte}, we interpolate to get
\begin{equation}
\norm{\nabla^\ell u(t)}_{L^{2}}\leq \norm{\nabla^{-s} u(t)}_{L^{2}}^{\frac{1%
}{\ell +1+s}}\norm{\nabla^{\ell+1} u(t)}_{L^{2}}^{\frac{\ell +s}{\ell +1+s}}.
\label{heat 3}
\end{equation}%
Combining \eqref{heat 3} and \eqref{heat 2} (with $\ell =-s$), we obtain
\begin{equation}
\norm{\nabla^{\ell+1} u(t)}_{L^{2}}\geq \norm{\nabla^{-s} u_0}_{L^{2}}^{-%
\frac{1}{\ell +s}}\norm{\nabla^\ell u(t)}_{L^{2}}^{1+\frac{1}{\ell +s}}.
\label{heat 4}
\end{equation}%
Plugging \eqref{heat 4} into \eqref{heat 1}, we deduce that there exists a
constant $C_{0}>0$ such that
\begin{equation}
\frac{d}{dt}\norm{\nabla^\ell u}_{L^{2}}^{2}+C_{0}\left( \norm{\nabla^\ell u}%
_{L^{2}}^{2}\right) ^{1+\frac{1}{\ell +s}}\leq 0.  \label{heat 5}
\end{equation}%
Solving this inequality directly, we obtain
\begin{equation}
\norm{\nabla^\ell u(t)}_{L^{2}}^{2}\leq \left( \norm{\nabla^\ell u_0}%
_{L^{2}}^{-\frac{2}{\ell +s}}+\frac{C_{0}t}{\ell +s}\right) ^{-(\ell
+s)}\leq C_{0}(1+t)^{-(\ell +s)}.  \label{heat 5''}
\end{equation}%
Thus we deduce \eqref{heat decay} by taking the square root of \eqref{heat
5''}.
\end{proof}

\begin{Remark}
The general optimal $L^q$ decay rates of the solution follow by \eqref{heat
decay} and the Sobolev interpolation (cf. Lemma \ref{1interpolation}). For
instance,
\begin{equation}  \label{Linfty heat decay}
\norm{u(t)}_{L^\infty}\le C\norm{u(t)}_{L^2}^{\frac{1}{4}}\norm{\nabla^2
u(t)}_{L^2}^{\frac{3}{4}}\le C_0(1+t)^{-\left(\frac{3}{4}+\frac{s}{2}%
\right)}.
\end{equation}
An important feature in Theorem \ref{heat} is that the $\dot{H}^{-s}$ norm
of the solution is preserved along time evolution. Compared to \eqref{decay
rate}, the result \eqref{heat decay} demonstrates that the $\dot{H}%
^{-s}(s\ge 0)$ norm of initial data enhances the decay rate of the solution
with the factor $s/2$.
\end{Remark}

\begin{Remark}
Although Theorem \ref{heat} can be proved by the Fourier analysis or
spectral method, the same strategy in our proof can be applied to nonlinear
system with two essential points in the proof: (1) closing the energy
estimates at each $\ell $-th level (referring to the order of the spatial
derivatives of the solution); (2) deriving a novel negative Sobolev
estimates for nonlinear equations which requires $s<3/2$ ($n/2$ for dimension $n$).
\end{Remark}

The rest of this paper is organized as follows. In Section \ref{s2}, we will
state our main results of this paper. We will prove the main theorem of the
compressible Navier-Stokes equations in Section 3 and prove the main
theorems of the Boltzmann equation in Section 4, respectively. The analytic
tools used in this paper will be collected in Appendix. We point out here that our method can be applied to many dissipative equations in the whole space. For example, the natural extension of this paper is to considering the compressible Navier-Stokes equations and the Boltzmann equation under the influence of the self-consistent electric field or electromagnetic field, and these will be reported in the forthcoming papers.


\section{Main results}

\label{s2} 


\subsection{Main results for compressible Navier-Stokes equations}


Considering the compressible Navier-Stokes equations
\begin{equation}  \label{1NS}
\left\{%
\begin{array}{lll}
\partial_t\rho+\mathrm{div}(\rho u )=0 &  &  \\
\partial_t(\rho u) +\mathrm{div}(\rho u \otimes u )+\nabla p(\rho)-\mu\Delta
u -(\mu+\lambda)\nabla\mathrm{div} u =0 &  &  \\
(\rho,u)|_{t=0}=(\rho_0, u _0), &  &
\end{array}%
\right.
\end{equation}
which governs the motion of a compressible viscous fluid. Here $t\ge0$ and $%
x\in \mathbb{R}^3$. The unknown functions $\rho, u $ represent the density,
velocity of the fluid respectively, and the pressure $p=p(\rho)$ is a smooth
function in a neighborhood of $\bar{\rho}$ with $p^{\prime}(\bar{\rho})>0$,
where $\bar{\rho}$ is a positive constant. We assume that the constant
viscosity coefficients $\mu$ and $\lambda$ satisfy the usual physical
conditions
\begin{equation}  \label{1viscosity}
\mu>0,\quad \lambda+\frac{2}{3}\mu\ge 0.
\end{equation}

The convergence rate of solutions of the Cauchy problem \eqref{1NS} to the
steady state has been investigated extensively since the first global
existence of small solutions in $H^{3}$ (classical solutions) was proved in
\cite{MN1}. For the initial perturbation small in $H^{3}\cap L^{1}$, \cite{MN2} obtained
\begin{equation}
\norm{(\rho-\bar{\rho},u)(t)}_{L^{2}}\lesssim (1+t)^{-3/4},
\end{equation}%
and for the small initial perturbation belongs to $H^{m}\cap W^{m,1}$ with $%
m\geq 4$, \cite{P} proved the optimal $L^{q}$ decay rate
\begin{equation}
\Vert \nabla ^{k}(\rho -\bar{\rho},u)(t)\Vert _{L^{q}}\lesssim (1+t)^{-\frac{%
3}{2}\left( 1-\frac{1}{q}\right) -\frac{k}{2}},\ \hbox{ for }2\leq p\leq
\infty \hbox{ and }0\leq k\leq 2.
\end{equation}%
By the detailed study of the Green function, the optimal $L^{q},1\leq q\leq
\infty $ decay rates were also obtained in \cite{HZ1,HZ2,LW} for the small
initial perturbation belongs to $H^{m}\cap L^{1}$ with $m\geq 4$. These
results were extended to the exterior problem \cite{KS1,K1} or the half
space problem \cite{KK1,KK2} or with an external potential force \cite{DUYZ1}%
, but without the smallness of $L^{1}$-norm of the initial perturbation. For
the small initial perturbation belongs to $H^{3}$ only, by a weighted energy
method, \cite{M} showed the optimal decay rates
\begin{equation}
\norm{\nabla^k(\rho-\bar{\rho},u)(t)}_{L^{2}}\lesssim (1+t)^{-k/2}\hbox{ for
}k=1,2,\hbox{ and }\norm{(\rho-\bar{\rho},u)(t)}_{L^{\infty }}\lesssim
(1+t)^{-3/4};  \label{1Mat}
\end{equation}%
While based on a differential inequality, \cite{D1,D2} obtained a slower
(than the optimal) decay rate for the problem in unbounded domains with
external force through the pure energy method.

We will apply the energy method illustrated in Theorem \ref{heat} to prove
the $L^2$ optimal decay rate of the solution to the problem \eqref{1NS}. We
rewrite \eqref{1NS} in the perturbation form as
\begin{equation}  \label{1NS2}
\left\{%
\begin{array}{lll}
\displaystyle\partial_t\varrho +\bar{\rho}\mathrm{div}u=-\varrho\mathrm{div}%
u-u\cdot\nabla\varrho &  &  \\
\displaystyle\partial_tu -\bar{\mu}\Delta u -(\bar{\mu}+\bar{\lambda})\nabla%
\mathrm{div}u+\gamma\bar{\rho}\nabla\varrho =-u\cdot\nabla u
-h(\varrho)\left(\bar{\mu}\Delta u+(\bar{\mu}+\bar{\lambda})\nabla\mathrm{div%
} u\right)-f(\varrho)\nabla\varrho &  &  \\
(\varrho,u)|_{t=0}=(\varrho_0, u _0), &  &
\end{array}%
\right.
\end{equation}
where $\varrho=\rho-\bar{\rho}$, $\bar{\mu}=\mu/\bar{\rho},\ \bar{\lambda}%
=\lambda/\bar{\rho},\ \gamma={p^{\prime}(\bar{\rho})}/{\bar{\rho}^2}$, and
the two nonlinear functions of $\varrho$ are defined by
\begin{equation}  \label{1 h and f}
h(\varrho):=\frac{\varrho}{\varrho+\bar{\rho}} \hbox{ and } f(\varrho):=%
\frac{p^{\prime}(\varrho+\bar{\rho})}{{\varrho+\bar{\rho}}}-\frac{p^{\prime}(%
\bar{\rho})}{\bar{\rho}}.
\end{equation}
Then our main results are stated in the following theorem:

\begin{theorem}
\label{1mainth} Assume that $(\varrho _{0},u_{0})\in H^{N}$ for an integer $%
N\geq 3$. Then there exists a constant $\delta _{0}>0$ such that if
\begin{equation}
\norm{\varrho_0}_{H^{\left[ \frac{N}{2}\right] +2}}+\norm{u_0 }_{H^{\left[
\frac{N}{2}\right] +2}}\leq \delta _{0},  \label{1Hn/2}
\end{equation}%
then the problem \eqref{1NS2} admits a unique global solution $(\varrho
(t),u(t))$ satisfying that for all $t\geq 0$,
\begin{equation}
\norm{ \varrho(t)}_{H^{m}}^{2}+\norm{ u(t)}_{H^{m}}^{2}+\int_{0}^{t}%
\norm{\nabla\varrho(\tau)}_{H^{m-1}}^{2}+\norm{\nabla u(\tau)}%
_{H^{m}}^{2}\,d\tau \leq C\left( \norm{ \varrho_0}_{H^{m}}^{2}+\norm{ u_0}%
_{H^{m}}^{2}\right) ,  \label{1HNbound}
\end{equation}%
where $\left[ \frac{N}{2}\right] +2\leq m\leq N$. If further, $\varrho
_{0},u_{0}\in \dot{H}^{-s}$ for some $s\in \lbrack 0,3/2)$, then for all $%
t\geq 0$,
\begin{equation}
\norm{\Lambda^{-s} \varrho(t)}_{L^{2}}^{2}+\norm{\Lambda^{-s} u(t)}%
_{L^{2}}^{2}\leq C_{0}  \label{1H-sbound}
\end{equation}%
and
\begin{equation}
\norm{\nabla^\ell \varrho(t)}_{H^{N-\ell }}^{2}+\norm{\nabla^\ell u(t)}%
_{H^{N-\ell }}^{2}\leq C_{0}(1+t)^{-(\ell +s)}\ \hbox{ for }-s<\ell \leq N-1.
\label{1decay}
\end{equation}
\end{theorem}

\begin{Remark}
For $N=3$ and $s=0$, our decay rates \eqref{1decay} coincide with %
\eqref{1Mat} of \cite{M}. While for the other cases, our results are
completely new. Notice that we do not assume that $\dot{H}^{-s}$ norm of initial data is small and this norm enhances the decay rate of the solution to be faster than that of \cite{M}. The constraint $s<3/2$ comes
from applying Lemma \ref{1Riesz} to estimate the nonlinear terms when doing
the negative Sobolev estimates via $\Lambda ^{-s}$. For $s\geq 3/2$, the
nonlinear estimates would not work.
\end{Remark}

\begin{Remark}
Notice that we only assume that the lower order Sobolev norm of initial data is small, while the higher order Sobolev norm can be arbitrarily large. Although one may replace the range of smallness $[\frac{N}{2}]+2$ by a smaller number (e.g. $3$) by refining the energy estimates, this is beyond our primary interest in this paper.
\end{Remark}

The proof of Theorem \ref{1mainth} will be presented in Section 3, which is
inspired by the proof of Theorem \ref{heat}. However, we will be
not able to close the energy estimates at each $\ell $-th level as the heat
equation. This is essentially caused by the ``degenerate" dissipative structure of the linear homogenous system of \eqref{1NS2} when using our energy method. More precisely, the linear energy identity of the problem reads as: for $\ell=0,\dots,N$,
\begin{equation}\label{1energy identity}
\frac{1}{2}\frac{d}{dt}\int_{\mathbb{R}^3}\gamma|\nabla^\ell \varrho|^2+|\nabla^\ell u|^2\,dx
+\int_{\mathbb{R}^3}\bar{\mu}|\nabla \nabla^\ell u|^2+(\bar{\mu}+\bar{\lambda})|\diverge \nabla^\ell u|^2\,dx=0.
\end{equation}
The constraint \eqref{1viscosity} implies that there exists a constant $\sigma_0>0$ such that
\begin{equation}\label{1coercive}
\int_{\mathbb{R}^3}\bar{\mu}|\nabla \nabla^\ell u|^2+(\bar{\mu}+\bar{\lambda})|\diverge \nabla^\ell u|^2\,dx
\ge \sigma_0 \norm{\nabla^{\ell+1} u}_{L^2}^2.
\end{equation}
Note that \eqref{1energy identity} and \eqref{1coercive} only give the dissipative estimate for $u$. To rediscover the dissipative estimate for $\varrho$, we will use the linearized equations of $\eqref{1NS2}$ via constructing the interactive energy functional between $u$ and $\nabla\rho$ to deduce
\begin{equation}\label{1daf}
\frac{d}{dt}\int_{\mathbb{R}^3} \nabla^\ell u\cdot\nabla\nabla^\ell \varrho\,dx +C%
\norm{\nabla^{\ell+1}\varrho}_{L^2}^2 \lesssim\norm{\nabla^{\ell+1}u}_{L^2}^2+%
\norm{\nabla^{\ell+2}u}_{L^2}^2.
\end{equation}
This implies that to get the dissipative estimate for $\nabla^{\ell+1}\varrho$ it requires us to do the energy estimates \eqref{1energy identity} at both the $\ell$-th and the $\ell+1$-th levels (referring to the order of the spatial derivatives of the solution). To get around this obstacle, the idea is to construct some energy functionals ${\mathcal{E}}%
_{\ell }^{m}(t),\ \left[ \frac{N}{2}\right] +2\leq
m\leq N$ and $0\leq \ell \leq m-1$ (less than $m-1$ is restricted by \eqref{1daf}),
\begin{equation}
{\mathcal{E}}_{\ell }^{m}(t)\backsim \sum_{\ell \leq k\leq m}%
\norm{\left[\nabla^k \varrho(t),\nabla^k u(t)\right]}_{L^2}^{2},  \notag
\end{equation}%
which has a \textit{minimum} derivative count $\ell .$ We will then close the energy estimates at each $\ell $-th level in a weak sense by deriving
the Lyapunov-type inequality (cf. \eqref{1proof5}) for these energy
functionals in which the corresponding dissipation (denoted by ${\mathcal{D}}%
_{\ell }^{m}(t)$) can be related to the energy ${\mathcal{E}}_{\ell }^{m}(t)$
similarly as \eqref{heat 4} by the Sobolev interpolation. This can be easily
established for the linear homogeneous problem along our analysis, however,
for the nonlinear problem \eqref{1NS2}, it is much more complicated due to
the nonlinear estimates. This is the second point of this paper that we will
extensively and carefully use the Sobolev interpolation of the
Gagliardo-Nirenberg inequality between high-order and low-order spatial
derivatives to bound the nonlinear terms by $\sqrt{\mathcal{E}_{0}^{\left[
N/2\right] +2}(t)}{\mathcal{D}}_{\ell }^{m}(t)$ that can be absorbed. When deriving the negative Sobolev
estimates, we need to restrict that $s<3/2$ in order to estimate $\Lambda ^{-s}$ acting on the
nonlinear terms by using the Hardy-Littlewood-Sobolev inequality, and also
we need to separate the cases that $s\in (0,1/2]$ and $s\in (1/2,3/2)$. Once
these estimates are obtained, Theorem \ref{1mainth} follows by the
interpolation between negative and positive Sobolev norms similarly as in
the proof of Theorem \ref{heat}.


\subsection{Main results for Boltzmann equation}


The dynamics of dilute particles can be described by the Boltzmann equation:
\begin{equation}
\partial_tF+v\cdot\nabla_xF=Q(F,F),
\end{equation}
with initial data $F(0,x,v)=F_0(x,v)$. Here $F=F(t,x,v)\ge 0$ is the number
density function of the particles at time $t\ge 0$, position $%
x=(x_1,x_2,x_3)\in \mathbb{R}^3$ and velocity $v=(v_1,v_2,v_3)\in \mathbb{R}%
^3$. The collision between particles is given by the standard Boltzmann
collision operator $Q(h_1,h_2)$ with hard-sphere interaction:
\begin{equation}  \label{Boltzmann operator}
Q(h_1,h_2)(v)=\int_{\mathbb{R}^3}\int_{\mathbb{S}^2}|(u-v)\cdot\omega|%
\{h_1(v^{\prime})h_2(u^{\prime})-h_1(v)h_2(u)\}\, d\omega\, du.
\end{equation}
Here $\omega\in \mathbb{S}^2$, and
\begin{equation}
v^{\prime}=v-[(v-u)\cdot\omega]\omega,\quad
u^{\prime}=u+[(v-u)\cdot\omega]\omega,
\end{equation}
which denote velocities after a collision of particles having velocities $v$%
, $u$ before the collision and vice versa.

We denote a normalized global Maxwellian by
\begin{equation}
\mu(v)= \mathrm{e}^{-|v|^2/2},
\end{equation}
and define the standard perturbation $f(t,x,v)$ to $\mu$ as
\begin{equation}
F=\mu+\sqrt{\mu}f.
\end{equation}
The Boltzmann equation for the perturbation $f$ now takes the form
\begin{equation}  \label{perturbation}
\partial_tf + v\cdot\nabla_xf + L f=\Gamma(f,f),
\end{equation}
with initial data $f(0,x,v)=f_0(x,v)$. Here the linearized collision
operator $L$ is given by
\begin{equation}
L h=-\frac{1}{\sqrt{\mu}}\{Q(\mu,\sqrt{\mu}h)+Q(\sqrt{\mu}h,\mu)\},
\end{equation}
and the nonlinear collision operator (non-symmetric) is
\begin{equation}
\Gamma(h_1,h_2)=\frac{1}{\sqrt{\mu}}Q(\sqrt{\mu}h_1,\sqrt{\mu}h_2).
\end{equation}

It is well-known that the operator $L\ge 0$, and for any fixed $(t,x)$, the
null space of $L$ is
\begin{equation}
\mathcal{N}=\mathrm{span}\left\{\sqrt{\mu},v\sqrt{\mu},|v|^2\sqrt{\mu}%
\right\}.
\end{equation}
For any fixed $(t,x)$, we define $\mathbf{P}$ as the $L^2_v$ orthogonal
projection on the null space $\mathcal{N}$. Thus for any function $f(t,x,v)$
we can decompose
\begin{equation}
f=\mathbf{P}f+\{\mathbf{I-P}\}f.
\end{equation}
Here $\mathbf{P}f$ is called the hydrodynamic part of $f$, and $\{\mathbf{I-
P}\}f$ is the microscopic part.

\noindent \textbf{Notation 2.} In the context of the Boltzmann equation, we
shall use $\langle \cdot ,\cdot \rangle $ to denote the $L^{2}$ inner
product in $\mathbb{R}_{v}^{3}$ with corresponding $L^{2}$ norm $|\cdot
|_{2}$, while we use $(\cdot ,\cdot )$ to denote the $L^{2}$ inner
product either in $\mathbb{R}_{x}^{3}\times \mathbb{R}_{v}^{3}$ or in $%
\mathbb{R}_{x}^{3}$ with $L^{2}$ norm $\Vert \cdot \Vert _{L^{2}}$ without
any ambiguity. We shall simply use $L_{x}^{2}$, $L_{v}^{2}$ to denote $L^{2}(%
\mathbb{R}_{x}^{3})$ and $L^{2}(\mathbb{R}_{v}^{3})$ respectively, etc. We
will use the notation $L_{v}^{2}H_{x}^{s}$ to denote the space $L^{2}(%
\mathbb{R}_{v}^{3};H_{x}^{s})$ with norm
\begin{equation}
\norm{h}_{L_{v}^{2}H_{x}^{s}}=\left( \int_{\mathbb{R}_{v}^{3}}\norm{f}%
_{H_{x}^{s}}^{2}\,dv\right) ^{1/2},
\end{equation}%
and similarly we use the notations of $L_{v}^{2}\dot{H}_{x}^{s}$, $%
L_{v}^{2}L_{x}^{p}$ and $L_{x}^{p}L_{v}^{2}$, etc. For the Boltzmann
operator \eqref{Boltzmann operator}, we define the collision frequency as
\begin{equation}
\nu (v)=\int_{\mathbb{R}^{3}}|v-u|\mu (u)du,
\end{equation}%
which behaves like $1+|v|$. We define the weighted $L^{2}$ norms
\begin{equation}
|g|_{\nu }^{2}=|\nu ^{1/2}g|_{2}^{2},\quad \Vert g\Vert _{\nu }^{2}=\Vert
\nu ^{1/2}g\Vert ^{2}.
\end{equation}%
We denote $L_{\nu }^{2}$ by the weighted space with norm $\Vert \cdot \Vert
_{\nu }$.

We will apply the energy method illustrated in Theorem \ref{heat} to prove
the $L^{2}$ optimal decay rate of the solution to the problem %
\eqref{perturbation}. Main results are stated in the following theorems.

\begin{theorem}
\label{theorem1} Assume that $f_{0}\in L_{v}^{2}H_{x}^{N}$ for an integer $%
N\geq 3$. Then there exists a constant $\delta _{0}>0$ such that if
\begin{equation}
\sum_{0\leq k\leq N}\norm{\nabla^k f_0}_{L^2}\leq \delta _{0},  \label{initial ass}
\end{equation}%
then the problem \eqref{perturbation} admits a unique global solution $%
f(t,x,v)$ satisfying that for all $t\geq 0$,
\begin{equation}
\sum_{0\leq k\leq N}\norm{\nabla^k f(t)}_{L^2}^{2}+\int_{0}^{t}\norm{ \{{\bf
I-P}\} f(\tau)}_{\nu }^{2}+\sum_{1\leq k\leq N}\norm{\nabla^k f(\tau)}_{\nu
}^{2}\,d\tau \leq C\sum_{0\leq k\leq N}\norm{\nabla^k f_0}_{L^2}^{2}.
\label{energy es}
\end{equation}%
If further, $f_{0}\in L_{v}^{2}\dot{H}_{x}^{-s}$ for some $s\in \lbrack
0,3/2)$, then for all $t\geq 0$,
\begin{equation}
\norm{\Lambda^{-s}f(t)}_{L^2}\leq C_{0}  \label{H-sbound}
\end{equation}%
and
\begin{equation}
\sum_{\ell \leq k\leq N}\norm{\nabla^k f(t)}_{L^2}^{2}\leq C_{0}(1+t)^{-(\ell
+s)}\,\hbox{ for }-s<\ell \leq 1,  \label{decay1}
\end{equation}%
and
\begin{equation}
\norm{\{{\bf I-P}\} f(t)}_{L^2}^{2}\leq C_{0}(1+t)^{-(1+s)}.  \label{decay1'}
\end{equation}

Furthermore,  if
\begin{equation}
\sum_{0\leq k\leq N}\norm{\nabla^k f_0}_{L^2}+\norm{ f_0}_{\nu }\leq \delta _{0},
\label{initial ass2}
\end{equation}%
then for all $t\geq 0$,
\begin{equation}
\sum_{\ell \leq k\leq N}\norm{\nabla^k f(t)}_{L^2}^{2}\leq C_{0}(1+t)^{-(\ell
+s)}\,\hbox{ for }-s<\ell \leq N-1,  \label{decay2}
\end{equation}%
and
\begin{equation}
\norm{\nabla^\ell\{{\bf I-P}\} f(t)}_{L^2}^{2}\leq C_{0}(1+t)^{-(\ell +1+s)}\,%
\hbox{ for }-s<\ell \leq N-2.  \label{decay2'}
\end{equation}
\end{theorem}

\begin{Remark}
Notice that similarly to that of the compressible Navier-Stokes equations we
do not need to assume that $L^2_v\dot{H}^{-s}_x$ norm of initial data is
small and this norm enhances the decay rate of the solution. The constraint $%
s<3/2$ also comes from applying Lemma \ref{1Riesz} to estimate the nonlinear
terms when doing the negative Sobolev estimates via $\Lambda^{-s}$.
\end{Remark}

The proof of Theorems \ref{theorem1} will be presented in Section 4, which
is also inspired by the proof of Theorem \ref{heat}. However, similarly to the compressible Navier-Stokes equations, we will be
not able to close the energy estimates at each $\ell $-th level and this is caused by the ``degenerate" dissipative structure of the linear homogenous system of \eqref{perturbation} when using our energy method. More precisely, the linear energy identity of the problem reads as:
\begin{equation}\label{energy identity boltzmann}
\frac{1}{2}\frac{d}{dt}\norm{\nabla^\ell f}_{L^2}^2+(L\nabla^\ell f,\nabla^\ell f)=0.
\end{equation}
It is well-known that $L$ is only positively definite with respect to the microscopic part $\{{\bf I-P}\}f$, that is, there exists a constant $\sigma_0>0$ such that
\begin{equation}\label{coercive boltzmann}
(L \nabla^\ell f, \nabla^\ell f)\ge \sigma_0 \norm{\nabla^\ell\{{\bf I-P}\} f}_\nu^2.
\end{equation}
To rediscover the dissipative estimate for the hydrodynamic part
${\bf P}f$, we will use the linearized equation of \eqref{perturbation} via constructing the interactive energy functional $G_\ell$ between $\nabla^\ell f$ and $\nabla^{\ell+1} f$ to deduce
\begin{equation}\label{3232}
\frac{dG_\ell}{dt}+\norm{\nabla^{\ell+1} \mathbf{P }f }_{L^2}^2 \lesssim  \norm{\nabla^{\ell}\{%
\mathbf{I-P }\} f  }_{L^2}^2+\norm{\nabla^{\ell+1}\{\mathbf{I-P }\} f  }_{L^2}^2.
\end{equation}
This implies that to get the dissipative estimate for $\nabla^{\ell+1}{\bf P}f$ it requires us to do the energy estimates \eqref{energy identity boltzmann} at both the $\ell$-th and the $\ell+1$-th levels (referring to the order of the spatial derivatives of the solution). To get around this obstacle, the idea is to construct some energy functionals ${\mathcal{E}}_{\ell }(t)$, $0\leq \ell \leq N-1$ (less than $N-1$ is restricted by \eqref{3232}),
\begin{equation}
{\mathcal{E}}_{\ell }(t)\backsim \sum_{\ell \leq k\leq N}\norm{\nabla^k f(t)}_{L^2}%
^{2},  \notag
\end{equation}%
which has a \textit{minimum} derivative count of \ $\ell ,$ and we will
derive the Lyapunov-type inequalities (cf. \eqref{energy estimate 5} and %
\eqref{energy estimate 9}) for these energy functionals in which the
corresponding dissipation (denoted by ${\mathcal{D}}_{\ell }(t)$) can be
related to the energy ${\mathcal{E}}_{\ell }(t)$ similarly as \eqref{heat 4}
by the Sobolev interpolation. This can be easily established for the linear
homogeneous problem along our analysis, however, for the nonlinear problem %
\eqref{perturbation}, we shall use extensively the Sobolev interpolation of
the the Gagliardo-Nirenberg inequality (for the functions defined on $%
\r3_{x}\times \r3_{v}$) between high-order and low-order
spatial derivatives to expect to bound the nonlinear terms by $\sqrt{%
\mathcal{E}_{0}(t)}{\mathcal{D}}_{\ell }(t)$ that can be absorbed. But this
can not be achieved well at this moment and we will be left with one extra
term related to a sum of velocity-weighted norms of $f$, as stated in %
\eqref{energy estimate 5}. Note that when taking $\ell =0,1$ in %
\eqref{energy estimate 5}, we can absorb this unpleasant term.
While for $\ell \geq 2$, we need to assume the weighted norm of the initial
data. With the help of these weighted norms, we will succeed in removing
this sum of velocity-weighted norms from the right hand side of %
\eqref{energy estimate 5} to get \eqref{energy estimate 9} in which we can
take $\ell =2,\dots ,N-1$. To estimate the negative Sobolev norm in Lemma %
\ref{lemma H-s}, we need to restrict that $s<3/2$ when estimating $\Lambda
^{-s}$ acting on the nonlinear terms, and we also need to separate the cases
that $s\in (0,1/2]$ and $s\in (1/2,3/2)$. We remark that it is also
important that we use the Minkowski's integral inequality to exchange the
order of integrations in $v$ and $x$ in order to estimate the nonlinear
terms and that we extensively use the splitting $f=\mathbf{P}f+{\{\mathbf{I-P%
}}\}f$.

We end this subsection by reviewing some previous related works on the
global existence and the time decay rates of solutions to the Boltzmann
equation. The existence of global solutions near Maxwellians has been
established in various function spaces, see \cite%
{U1974,NI1976,S1983,G2004,LYY2004,UY2006,G2006} for instance. It was
also shown in \cite{U1974,S1983,G2006} that the solutions in the
periodic domain or bounded domain decay in time at the exponential rate and
in \cite{NI1976,UY2006} that the solutions in the whole space decay at the
optimal algebraic rate of $(1+t)^{-3/4}$ if additionally the initial
perturbation is small in $L_{v}^{2}L_{x}^{1}$. On the other hand, some
analogous theorems of global existence and decay rate of the solutions to
the Boltzmann equation with forces have also been established; see \cite%
{UYZ2005,UYZ,DUYZ2} for the Boltzmann equation with external forces, \cite%
{G2002,YZ,YY,DS1} for the Vlasov-Poisson-Boltzmann system and \cite%
{G2003,S2006,J2009,DS2} for the Vlasov-Maxwell-Boltzmann system. However,
among these references the optimal decay rates of the solution have been
only established under the additional assumption that the initial
perturbation is small in $L_{v}^{2}L_{x}^{1}$. Based on the techniques of
using an time differential inequality in \cite{D1} and the pure energy
method, \cite{UYZ} and \cite{YZ} obtained the convergence rates (but slower
than the optimal rates) for the Boltzmann equation with external potential
force and the Vlasov-Poisson-Boltzmann system respectively.


\section{Compressible Navier-Stokes equations}

\label{sec 3} 

\subsection{Energy estimates}

\label{sec 3.1} 

In this subsection, we will derive the a priori nonlinear energy estimates
for the system \eqref{1NS2}. Hence we assume a priori that for sufficiently
small $\delta>0$,
\begin{equation}  \label{1a priori}
\norm{\varrho(t)}_{H^{\left[\frac{N}{2}\right]+2}}+\norm{u(t)}_{H^{\left[%
\frac{N}{2}\right]+2}}\le \delta.
\end{equation}

First of all, by \eqref{1a priori} and Sobolev's inequality, we obtain
\begin{equation}
{\bar{\rho}}/{2}\le \varrho+\bar{\rho}\le 2\bar{\rho}.
\end{equation}
Hence, we immediately have
\begin{equation}
|h(\varrho)|,|f(\varrho)|\le C|\varrho|\hbox{ and } |h^{(k)}(%
\varrho)|,|f^{(k)}(\varrho)| \le C\hbox{ for any }k\ge 1.  \label{1hf}
\end{equation}
where $h$ and $f$ are nonlinear functions of $\varrho$ defined by \eqref{1 h
and f}. Next, to estimate the $L^\infty$ norm of the spatial derivatives of $%
h$ and $f$, we shall record the following estimate:

\begin{lemma}
\label{1composition} Assume that $\norm{\varrho}_{H^2}\le 1$. Let $g(\varrho)
$ be a smooth function of $\varrho$ with bounded derivatives, then for any
integer $m\ge1$ we have
\begin{equation}
\norm{\nabla^m(g(\varrho))}_{L^\infty} \lesssim\norm{\nabla^m\varrho}%
_{L^2}^{1/4}\norm{\nabla^{m+2}\varrho}_{L^2} ^{3/4}.
\end{equation}
\end{lemma}

\begin{proof}
Notice that for $m\ge 1$,
\begin{equation}
\nabla^m(g(\varrho))=\hbox{ a sum of products }g^{\gamma_1,\dots,\gamma_n}(%
\varrho)\nabla^{\gamma_1}\varrho\cdots\nabla^{\gamma_n}\varrho,
\end{equation}
where the functions $g^{\gamma_1,\dots,\gamma_n}(\varrho)$ are some
derivatives of $g(\varrho)$ and $1\le \gamma_i\le m,\ i=1,\dots,n$ with $%
\gamma_1+\cdots+\gamma_n=m$. We then use the Sobolev interpolation of Lemma %
\ref{1interpolation} to bound
\begin{equation}
\begin{split}
\norm{\nabla^m(g(\varrho))}_{L^\infty} &\lesssim\norm{\nabla^{\gamma_1}%
\varrho}_{L^\infty}\cdots\norm{\nabla^{\gamma_n}\varrho}_{L^\infty} \\
&\lesssim\left(\norm{\nabla^{\gamma_1}\varrho}_{L^2}\cdots%
\norm{\nabla^{\gamma_n}\varrho}_{L^2}\right)^{1/4}\left(\norm{\nabla^2%
\nabla^{\gamma_1}\varrho}_{L^2}\cdots\norm{\nabla^2\nabla^{\gamma_n}\varrho}%
_{L^2}\right)^{3/4} \\
&\lesssim\left(\norm{ \varrho}_{L^2}^{1-\gamma_1/m}\norm{\nabla^{m}\varrho}%
_{L^2}^{\gamma_1/m}\cdots\norm{ \varrho}_{L^2}^{1-\gamma_n/m}%
\norm{\nabla^{m}\varrho}_{L^2}^{\gamma_n/m}\right)^{1/4} \\
&\quad\times\left(\norm{\nabla^2 \varrho}_{L^2}^{1-\gamma_1/m} %
\norm{\nabla^{m+2}\varrho}_{L^2}^{\gamma_1/m}\cdots\norm{\nabla^2 \varrho}%
_{L^2}^{1-\gamma_n/m} \norm{\nabla^{m+2}\varrho}_{L^2}^{\gamma_n/m}%
\right)^{3/4} \\
&\lesssim\norm{\varrho}_{H^2}^{n-1}\norm{\nabla^m\varrho}_{L^2}^{1/4}%
\norm{\nabla^{m+2}\varrho}_{L^2} ^{3/4} .
\end{split}%
\end{equation}
Hence, we conclude our lemma since $\norm{\varrho}_{H^2}\le 1$.
\end{proof}

We begin with the first type of energy estimates including $\rho $ and $u$
themselves.

\begin{lemma}
\label{1Ekle} Assume that $0\le k\le N-1$, then we have
\begin{equation}  \label{1E_k}
\frac{d}{dt}\int_{\mathbb{R}^3} {\gamma} |\nabla^{k} \varrho|^2+|\nabla^{k}
u|^2\,dx+C\norm{\nabla^{k+1} u}_{L^2}^2 \lesssim \delta \left(%
\norm{\nabla^{k+1}\varrho}_{L^2}^2+\norm{\nabla^{k+1} u}_{L^2}^2\right).
\end{equation}
\end{lemma}

\begin{proof}
For $k=0$, multiplying $\eqref{1NS2}_{1},\eqref{1NS2}_{2}$ by $\gamma
\varrho ,u$ respectively, summing up and then integrating the resulting over
$\mathbb{R}^{3}$ by parts, by H\"{o}lder's and Sobolev's inequalities and
the fact \eqref{1hf}, we obtain
\begin{equation}
\begin{split}
& \frac{1}{2}\frac{d}{dt}\int_{\mathbb{R}^{3}}\gamma |\varrho
|^{2}+|u|^{2}\,dx+\int_{\mathbb{R}^{3}}\bar{\mu}|\nabla u|^{2}+(\bar{\mu}+%
\bar{\lambda})|\mathrm{div}u|^{2}\,dx \\
& \quad =\int_{\mathbb{R}^{3}}\gamma (-\varrho \mathrm{div}u-u\cdot \nabla
\varrho )\varrho -\left( u\cdot \nabla u+h(\varrho )(\bar{\mu}\Delta u+(\bar{%
\mu}+\bar{\lambda})\nabla \mathrm{div}u)+f(\varrho )\nabla \varrho \right)
\cdot u\,dx \\
& \quad \lesssim \norm{\varrho}_{L^{3}}\norm{\nabla u}_{L^{2}}\norm{\varrho}%
_{L^{6}}+\left( \norm{u}_{L^{3}}\norm{\nabla u}_{L^{2}}+\norm{\varrho}%
_{L^{6}}\norm{\nabla^2 u}_{L^{3}}+\norm{\varrho}_{L^{3}}\norm{\nabla \varrho}%
_{L^{2}}\right) \norm{u}_{L^{6}} \\
& \quad \lesssim \delta \left( \norm{\nabla \varrho}_{L^{2}}^{2}+%
\norm{\nabla u}_{L^{2}}^{2}\right) .
\end{split}
\label{1E_0}
\end{equation}%
By \eqref{1coercive}, we obtain %
\eqref{1E_k} for $k=0$.

Now for $1\le k\le N-1$, applying $\nabla^k$ to $\eqref{1NS2}_1, \eqref{1NS2}%
_2$ and then multiplying the resulting identities by $\gamma\nabla^k\varrho,
\nabla^k u$ respectively, summing up and integrating over $\mathbb{R}^3$, we
obtain
\begin{equation}  \label{1E_k_0}
\begin{split}
&\frac{1}{2}\frac{d}{dt}\int_{\mathbb{R}^3} \gamma|\nabla^k
\varrho|^2+|\nabla^k u|^2\,dx +\int_{\mathbb{R}^3}\bar{\mu}|\nabla^{k+1}
u|^2+(\bar{\mu}+\bar{\lambda})|\nabla^{k}\mathrm{div} u|^2\,dx \\
&\quad=\int_{\mathbb{R}^3}\gamma\nabla^k (-\varrho\mathrm{div}%
u-u\cdot\nabla\varrho)\nabla^k\varrho \\
&\qquad\qquad-\nabla^k\left(u\cdot\nabla u+h(\varrho)(\bar{\mu}\Delta u+(%
\bar{\mu}+\bar{\lambda})\nabla\mathrm{div} u)+f(\varrho)\nabla\varrho\right)%
\cdot\nabla^k u\,dx \\
&\quad=\int_{\mathbb{R}^3}\gamma\nabla^{k-1} (\varrho\mathrm{div}%
u+u\cdot\nabla\varrho)\nabla^{k+1}\varrho \\
&\qquad\qquad+\nabla^{k-1}\left(u\cdot\nabla u+h(\varrho)(\bar{\mu}\Delta u+(%
\bar{\mu}+\bar{\lambda})\nabla\mathrm{div} u)+f(\varrho)\nabla\varrho\right)%
\cdot\nabla^{k+1} u\,dx \\
&\quad:=I_1+I_2+I_3+I_4+I_5.
\end{split}%
\end{equation}

We shall estimate each term in the right hand side of \eqref{1E_k_0}. First,
for the term $I_1$, by H\"older's inequality and the Sobolev interpolation
of Lemma \ref{1interpolation}, we have
\begin{equation}  \label{1E_k_1_0}
\begin{split}
I_1&=\int_{\mathbb{R}^3}\gamma\nabla^{k-1} (\varrho\mathrm{div}%
u)\nabla^{k+1}\varrho\,dx \\
&=\gamma\int_{\mathbb{R}^3}\sum_{0\le \ell\le
k-1}C_{k-1}^\ell\nabla^{k-1-\ell} \varrho\nabla^\ell \mathrm{div}u
\nabla^{k+1}\varrho\,dx \\
&\lesssim \sum_{0\le \ell\le k-1}\norm{\nabla^{k-1-\ell}\varrho}_{L^\infty}%
\norm{\nabla^{\ell+1}u}_{L^2}\norm{ \nabla^{k+1}\varrho}_{L^2} \\
&\lesssim \sum_{0\le \ell\le k-1}\norm{\nabla^{k-1-\ell}\varrho}_{L^\infty}%
\norm{ u}_{L^2}^{1-\frac{\ell+1}{k+1}}\norm{\nabla^{k+1}u}_{L^2}^\frac{\ell+1%
}{k+1}\norm{ \nabla^{k+1}\varrho}_{L^2}.
\end{split}%
\end{equation}
The main idea is that we will carefully adjust the index in the right hand
side of \eqref{1E_k_1_0} so that it can be bounded by the right hand side of %
\eqref{1E_k}. This is the crucial point that helps us close our energy
estimates at each $k$-th level and avoid imposing the smallness of the whole
$H^N$ norm of initial data. To this end, we use Lemma \ref{1interpolation}
to do the interpolation
\begin{equation}  \label{1E_k_1_1}
\norm{\nabla^{k-\ell-1}\varrho}_{L^\infty} \lesssim \norm{\nabla^\alpha%
\varrho}_{L^2}^{\frac{\ell+1}{k+1}}\norm{\nabla^{k+1}\varrho}_{L^2}^{1-\frac{%
\ell+1}{k+1}},
\end{equation}
where $\alpha$ satisfies
\begin{equation}
\begin{split}
&\frac{k-\ell-1}{3}=\left(\frac{\alpha}{3}-\frac{1}{2}\right)\times\frac{%
\ell+1}{k+1}+\left(\frac{k+1}{3}-\frac{1}{2}\right)\times\left(1-\frac{\ell+1%
}{k+1}\right) \\
&\quad \Longrightarrow \alpha=\frac{ k+1 }{2(\ell+1) }\le \frac{k+1}{2}\le %
\left[\frac{N}{2}\right]+1.
\end{split}%
\end{equation}
Hence, plugging \eqref{1E_k_1_1} into \eqref{1E_k_1_0}, together with %
\eqref{1a priori} and Young's inequality, we obtain
\begin{equation}  \label{1E_k_1}
\begin{split}
I_{1}&\lesssim \sum_{0\le \ell\le k-1}\delta\norm{\nabla^{k+1}\varrho}%
_{L^2}^{1-\frac{\ell+1}{k+1}}\norm{\nabla^{k+1}u}_{L^2}^\frac{\ell+1}{k+1}%
\norm{ \nabla^{k+1}\varrho}_{L^2} \\
&\lesssim\delta\left(\norm{\nabla^{k+1}\varrho}_{L^2}^2+\norm{\nabla^{k+1}u}%
_{L^2}^2\right).
\end{split}%
\end{equation}

Similarly, we can bound
\begin{equation}  \label{1E_k_2}
\begin{split}
I_2&=\int_{\mathbb{R}^3}\gamma\nabla^{k-1}
(u\cdot\nabla\varrho)\nabla^{k+1}\varrho \,dx \\
&=\gamma\int_{\mathbb{R}^3}\sum_{0\le \ell\le
k-1}C_{k-1}^\ell\nabla^{k-1-\ell}u\cdot\nabla^\ell \nabla\varrho
\nabla^{k+1}\varrho\,dx \\
&\lesssim \sum_{0\le \ell\le k-1}\norm{\nabla^{k-1-\ell}u}_{L^\infty}\norm{
\nabla^{\ell+1}\varrho}_{L^2}\norm{\nabla^{k+1}\varrho}_{L^2} \\
&\lesssim\delta\left(\norm{\nabla^{k+1}\varrho}_{L^2}^2+\norm{\nabla^{k+1}u}%
_{L^2}^2\right)
\end{split}%
\end{equation}
and
\begin{equation}  \label{1E_k_3}
\begin{split}
I_3&=\int_{\mathbb{R}^3} \nabla^{k-1}\left(u\cdot\nabla u
\right)\cdot\nabla^{k+1} u\,dx \\
&=\gamma\int_{\mathbb{R}^3}\sum_{0\le \ell\le
k-1}C_{k-1}^\ell\nabla^{k-1-\ell} u\cdot\nabla^\ell \nabla u
\nabla^{k+1}\varrho\,dx \\
&\lesssim \sum_{0\le \ell\le k-1}\norm{\nabla^{k-1-\ell}u}_{L^\infty}%
\norm{\nabla^{\ell+1}u}_{L^2}\norm{ \nabla^{k+1}u}_{L^2} \\
&\lesssim \delta \norm{\nabla^{k+1}u}_{L^2}^2.
\end{split}%
\end{equation}

Next, we estimate the term $I_{4}$. First, we notice that
\begin{equation}
\begin{split}
I_{4}& =\int_{\mathbb{R}^{3}}\nabla ^{k-1}\left( h(\varrho )(\bar{\mu}\Delta
u+(\bar{\mu}+\bar{\lambda})\nabla \mathrm{div}u)\right) \cdot \nabla
^{k+1}u\,dx \\
& \approx \int_{\mathbb{R}^{3}}\nabla ^{k-1}\left( h(\varrho )\nabla
^{2}u\right) \cdot \nabla ^{k+1}u\,dx \\
& =\int_{\mathbb{R}^{3}}\sum_{0\leq \ell \leq k-1}C_{k-1}^{\ell }\nabla
^{k-1-\ell }h(\varrho )\nabla ^{\ell }\nabla ^{2}u\nabla ^{k+1}u\,dx \\
& \lesssim \sum_{0\leq \ell \leq k-1}\norm{\nabla^{k-1-\ell} h(\varrho)
\nabla^{\ell+2}u}_{L^{2}}\norm{ \nabla^{k+1}u}_{L^{2}}.
\end{split}
\label{1E_k_4_0}
\end{equation}%
We shall separate the cases in the summation of \eqref{1E_k_4_0}. For $\ell
=k-1$, we have
\begin{equation}
\norm{ h(\varrho) \nabla^{k+1}u}_{L^{2}}\norm{\nabla^{k+1}u}_{L^{2}}\lesssim %
\norm{ h(\varrho)}_{L^{\infty }}\norm{\nabla^{k+1}u}_{L^{2}}%
\norm{\nabla^{k+1}u}_{L^{2}}\lesssim \delta \norm{\nabla^{k+1}u}_{L^{2}}^{2};
\label{1E_k_4_1}
\end{equation}%
for $\ell =k-2$, by H\"{o}lder's and Sobolev's inequalities, we have
\begin{equation}
\norm{ \nabla (h(\varrho)) \nabla^{k}u}_{L^{2}}\norm{ \nabla^{k+1}u}%
_{L^{2}}\lesssim \norm{h'(\varrho)\nabla\varrho}_{L^{3}}\norm{ \nabla^{k}u}%
_{L^{6}}\norm{ \nabla^{k+1}u}_{L^{2}}\lesssim \delta \norm{\nabla^{k+1}u}%
_{L^{2}}^{2};  \label{1E_k_4_2}
\end{equation}%
and for $0\leq \ell \leq k-3$, noticing that $k-\ell \geq 3$, we may then
use Lemma \ref{1composition}, together with Lemma \ref{1interpolation}, to
bound
\begin{equation}
\begin{split}
\norm{\nabla^{k-1-\ell}(h(\varrho))}_{L^{\infty }}& \lesssim %
\norm{\nabla^{k-1-\ell}\varrho}_{L^{2}}^{1/4}\norm{\nabla^{k+1-\ell}\varrho}%
_{L^{2}}^{3/4} \\
& \lesssim \left( \norm{\varrho}_{L^{2}}^{1-\frac{k-1-\ell }{k+1}}%
\norm{\nabla^{k+1}\varrho}_{L^{2}}^{\frac{k-1-\ell }{k+1}}\right)
^{1/4}\left( \norm{\varrho}_{L^{2}}^{1-\frac{k+1-\ell }{k+1}}%
\norm{\nabla^{k+1}\varrho}_{L^{2}}^{\frac{k+1-\ell }{k+1}}\right) ^{3/4} \\
& \lesssim \norm{\varrho}_{L^{2}}^{\frac{2\ell +1}{2(k+1)}}%
\norm{\nabla^{k+1}\varrho}_{L^{2}}^{1-{\frac{2\ell +1}{2(k+1)}}}.
\end{split}
\label{1E_k_4_3}
\end{equation}%
Therefore, by \eqref{1E_k_4_3} and using Lemma \ref{1interpolation} again,
we obtain
\begin{equation}
\begin{split}
& \norm{\nabla^{k-1-\ell} h(\varrho) \nabla^{\ell+2}u}_{L^{2}}\norm{
\nabla^{k+1}u}_{L^{2}} \\
& \quad \lesssim \norm{\nabla^{k-1-\ell} h(\varrho)}_{L^{\infty }}%
\norm{\nabla^{\ell+2}u}_{L^{2}}\norm{\nabla^{k+1}u}_{L^{2}} \\
& \quad \lesssim \norm{\varrho}_{L^{2}}^{\frac{2\ell +1}{2(k+1)}}%
\norm{\nabla^{k+1}\varrho}_{L^{2}}^{1-{\frac{2\ell +1}{2(k+1)}}}%
\norm{\nabla^\alpha u}_{L^{2}}^{1-{\frac{2\ell +1}{2(k+1)}}}%
\norm{\nabla^{k+1}u}_{L^{2}}^{\frac{2\ell +1}{2(k+1)}}\norm{ \nabla^{k+1}u}%
_{L^{2}} \\
& \quad \lesssim \delta \left( \norm{\nabla^{k+1}\varrho}_{L^{2}}^{2}+%
\norm{\nabla^{k+1}u}_{L^{2}}^{2}\right) ,
\end{split}
\label{1E_k_4_5}
\end{equation}%
where we have denoted $\alpha $ by
\begin{equation}
\begin{split}
& \ell +2=\alpha \times \left( 1-{\frac{2\ell +1}{2(k+1)}}\right)
+(k+1)\times \frac{2\ell +1}{2(k+1)} \\
& \quad \Longrightarrow \alpha =\frac{3(k+1)}{2(k-\ell )+1}\leq \frac{3(k+1)%
}{7}\leq \left[ \frac{N}{2}\right] +1\,\text{ since }k-\ell \geq 3.
\end{split}
\label{1E_k_4_4}
\end{equation}%
In light of \eqref{1E_k_4_1}, \eqref{1E_k_4_2} and \eqref{1E_k_4_5}, we find
\begin{equation}
I_{4}\lesssim \delta \left( \norm{\nabla^{k+1}\varrho}_{L^{2}}^{2}+%
\norm{\nabla^{k+1}u}_{L^{2}}^{2}\right) .  \label{1E_k_4}
\end{equation}

Finally, it remains to estimate the last term $I_5$. First, we have
\begin{equation}  \label{1E_k_5_0}
\begin{split}
I_5&=\int_{\mathbb{R}^3} \nabla^{k-1}\left(
f(\varrho)\nabla\varrho\right)\cdot\nabla^{k+1} u\,dx \\
&=\int_{\mathbb{R}^3}\sum_{0\le \ell\le
k-1}C_{k-1}^\ell\nabla^{k-1-\ell}f(\varrho)\nabla^\ell \nabla\varrho
\nabla^{k+1}u\,dx \\
&\lesssim\sum_{0\le \ell\le k-1}\norm{\nabla^{k-1-\ell}
f(\varrho)\nabla^{\ell+1}\varrho}_{L^2}\norm{ \nabla^{k+1}u}_{L^2} .
\end{split}%
\end{equation}
We shall separate the cases in the summation of \eqref{1E_k_5_0}. For $%
\ell=k-1$, we have
\begin{equation}  \label{1E_k_5_1}
\norm{f(\varrho)\nabla^{k}\varrho}_{L^2}\norm{ \nabla^{k+1}u}_{L^2} \lesssim%
\norm{ \varrho}_{L^3}\norm{\nabla^{k}\varrho}_{L^6}\norm{\nabla^{k+1}u}%
_{L^2} \lesssim\delta\left(\norm{\nabla^{k+1}\varrho}_{L^2}^2+%
\norm{\nabla^{k+1}u}_{L^2}^2\right).
\end{equation}
For $0\le\ell\le k-2$, similarly as in \eqref{1E_k_4_3}--\eqref{1E_k_4_4},
by Lemma \ref{1composition} and Lemma \ref{1interpolation}, we may bound
\begin{equation}  \label{1E_k_5_3}
\begin{split}
&\norm{\nabla^{k-1-\ell} f(\varrho)\nabla^{\ell+1}\varrho}_{L^2}\norm{
\nabla^{k+1}u}_{L^2} \\
&\quad \lesssim\delta^{\frac{2\ell+1}{2(k+1)}}\norm{\nabla^{k+1}\varrho}%
_{L^2}^{1-{\frac{2\ell+1}{2(k+1)}}}\norm{\nabla^\alpha\varrho}_{L^2}^{1-{%
\frac{2\ell+1}{2(k+1)}}}\norm{\nabla^{k+1}\varrho}_{L^2}^{\frac{2\ell+1}{%
2(k+1)}}\norm{ \nabla^{k+1}u}_{L^2} \\
&\quad \lesssim\delta \left(\norm{\nabla^{k+1}\varrho}_{L^2}^2+%
\norm{\nabla^{k+1}u}_{L^2}^2\right),
\end{split}%
\end{equation}
where we have denoted $\alpha$ by
\begin{equation}
\begin{split}
&\ell+1=\alpha\times\left(1-{\frac{2\ell+1}{2(k+1)}}\right)+(k+1)\times\frac{%
2\ell+1}{2(k+1)} \\
&\quad \Longrightarrow \alpha=\frac{ k+1 }{2(k-\ell)}+1\le \frac{k+1 }{5}\le %
\left[\frac{N}{2}\right]+1\, \text{ since }k-\ell\ge 2.
\end{split}%
\end{equation}
In light of \eqref{1E_k_5_1} and \eqref{1E_k_5_3}, we find
\begin{equation}  \label{1E_k_5}
I_5 \lesssim \delta\left(\norm{\nabla^{k+1}\varrho}_{L^2}^2+%
\norm{\nabla^{k+1}u}_{L^2}^2\right).
\end{equation}

Summing up the estimates for $I_{1}\sim I_{5}$, $i.e.$, \eqref{1E_k_1}, %
\eqref{1E_k_2}, \eqref{1E_k_3}, \eqref{1E_k_4} and \eqref{1E_k_5}, we deduce %
\eqref{1E_k} for $1\leq k\leq N-1$.
\end{proof}

Next, we derive the second type of energy estimates excluding $\rho $ and $u$
themselves$.$

\begin{lemma}
\label{1Ek+1le} Assume that $0\le k\le N-1$, then we have
\begin{equation}  \label{1E_k+1}
\begin{split}
&\frac{d}{dt}\int_{\mathbb{R}^3} \gamma|\nabla^{k+1}
\varrho|^2+|\nabla^{k+1} u|^2\,dx+C\norm{\nabla^{k+2} u}_{L^2}^2 \\
&\quad\lesssim \delta \left(\norm{\nabla^{k+1}\varrho}_{L^2}^2+%
\norm{\nabla^{k+1}u}_{L^2}^2+\norm{\nabla^{k+2}u}_{L^2}^2\right).
\end{split}%
\end{equation}
\end{lemma}

\begin{proof}
Let $0\le k\le N-1$. Applying $\nabla^{k+1}$ to $\eqref{1NS2}_1, \eqref{1NS2}%
_2$ and multiplying by $\gamma\nabla^{k+1}\varrho$, $\nabla^{k+1} u$
respectively, summing up and then integrating over $\mathbb{R}^3$ by parts,
we obtain
\begin{equation}  \label{1E_k+1_0}
\begin{split}
&\frac{1}{2}\frac{d}{dt}\int_{\mathbb{R}^3} \gamma|\nabla^{k+1}
\varrho|^2+|\nabla^{k+1} u|^2\,dx +\int_{\mathbb{R}^3}\bar{\mu}|\nabla^{k+2}
u|^2+(\bar{\mu}+\bar{\lambda})|\nabla^{k+1}\mathrm{div}u|^2\,dx \\
&\quad=\int_{\mathbb{R}^3}\gamma\nabla^{k+1} (-\varrho\mathrm{div}%
u-u\cdot\nabla\varrho)\nabla^{k+1}\varrho \\
&\qquad\quad\ \, -\nabla^{k+1}\left(u\cdot\nabla u+h(\varrho)(\bar{\mu}%
\Delta u+(\bar{\mu}+\bar{\lambda})\nabla\mathrm{div} u)+f(\varrho)\nabla%
\varrho\right)\cdot\nabla^{k+1} u\,dx \\
&\quad=\int_{\mathbb{R}^3}\gamma\nabla^{k+1} (-\varrho\mathrm{div}%
u-u\cdot\nabla\varrho)\nabla^{k+1}\varrho \\
&\qquad\quad\ \, +\nabla^{k}\left(u\cdot\nabla u+h(\varrho)(\bar{\mu}\Delta
u+(\bar{\mu}+\bar{\lambda})\nabla\mathrm{div} u)+f(\varrho)\nabla\varrho%
\right)\cdot\nabla^{k+2} u\,dx \\
&\quad:=J_1+J_2+J_3+J_4+J_5.
\end{split}%
\end{equation}

We shall estimate each term in the right hand side of \eqref{1E_k+1_0}.
First, we split $J_1$ as:
\begin{equation}  \label{1E_k+1_1_0}
\begin{split}
J_1&=-\gamma\int_{\mathbb{R}^3}\nabla^{k+1} (\varrho\mathrm{div}%
u)\nabla^{k+1}\varrho\,dx \\
&=-\gamma\int_{\mathbb{R}^3} \Big( \varrho \nabla^{k+1}\mathrm{div}u
+C_{k+1}^1\nabla\varrho \nabla^{k}\mathrm{div}u +C_{k+1}^2\nabla^2\varrho
\nabla^{k-1}\mathrm{div}u \\
&\qquad\qquad\quad+\sum_{3\le\ell\le k+1}C_{k+1}^\ell\nabla^\ell\varrho
\nabla^{k+1-\ell}\mathrm{div}u\Big)\nabla^{k+1}\varrho\,dx \\
&:=J_{11}+J_{12}+J_{13}+J_{14}.
\end{split}%
\end{equation}
Hereafter, it it happens to be the case $\ell>k+1$, etc., then it means
nothing. By H\"older's, Sobolev's and Cauchy's inequalities, we obtain
\begin{equation}  \label{1E_k+1_1_1}
\begin{split}
J_{11}\lesssim\norm{\varrho}_{L^\infty}\norm{\nabla^{k+2} u}_{L^2}%
\norm{\nabla^{k+1}\varrho}_{L^2} \lesssim \delta\left (\norm{\nabla^{k+1}%
\varrho}_{L^2}^2+\norm{\nabla^{k+2}u}_{L^2}^2\right)
\end{split}%
\end{equation}
and
\begin{equation}  \label{1E_k+1_1_2}
\begin{split}
J_{12}\lesssim\norm{\nabla\varrho}_{L^3} \norm{\nabla^{k+1}u}_{L^6}%
\norm{\nabla^{k+1}\varrho}_{L^2} \lesssim \delta(\norm{\nabla^{k+1}\varrho}%
_{L^2}^2+\norm{\nabla^{k+2}u}_{L^2}^2),
\end{split}%
\end{equation}
and
\begin{equation}  \label{1E_k+1_1_3}
\begin{split}
J_{13}\lesssim\norm{\nabla^2\varrho}_{L^3} \norm{\nabla^{k}u}_{L^6}%
\norm{\nabla^{k+1}\varrho}_{L^2} \lesssim \delta\left(\norm{\nabla^{k+1}%
\varrho}_{L^2}^2+\norm{\nabla^{k+1}u}_{L^2}^2\right).
\end{split}%
\end{equation}
While for the last term $J_{14}$, noting that now $k+2-\ell\le k-1$, by
H\"older's inequality and Lemma \ref{1interpolation}, we obtain
\begin{equation}  \label{1E_k+1_1_5}
\begin{split}
J_{14}&\lesssim\sum_{2\le\ell\le k+1}\norm{\nabla^\ell\varrho}_{L^2}%
\norm{\nabla^{k+2-\ell}u}_{L^\infty}\norm{\nabla^{k+1}\varrho}_{L^2} \\
&\lesssim\sum_{2\le\ell\le k+1}\norm{\varrho}_{L^2}^{1-\frac{\ell}{k+1}}%
\norm{\nabla^{k+1}\varrho}_{L^2}^\frac{\ell}{k+1} \norm{\nabla^\alpha u}%
_{L^2}^{\frac{\ell}{k+1} }\norm{\nabla^{k+2}u}_{L^2}^{1-\frac{\ell}{k+1}}%
\norm{\nabla^{k+1}\varrho}_{L^2} \\
&\lesssim \delta\left(\norm{\nabla^{k+1}\varrho}_{L^2}^2+\norm{\nabla^{k+2}u}%
_{L^2}^2\right),
\end{split}%
\end{equation}
where we have denoted $\alpha$ by
\begin{equation}
\begin{split}
&\frac{k-\ell+2}{3}=\left(\frac{\alpha}{3}-\frac{1}{2}\right)\times\frac{\ell%
}{k+1} +\left(\frac{k+2}{3}-\frac{1}{2}\right)\times\left(1-\frac{\ell}{k+1}%
\right) \\
&\quad \Longrightarrow \alpha=\frac{ 3(k+1) }{2\ell }+1\le \frac{ k+3 }{2}%
\le \left[\frac{N}{2}\right]+2\, \text{ since }\ell\ge 3.
\end{split}%
\end{equation}
In light of \eqref{1E_k+1_1_1},\ \eqref{1E_k+1_1_2},\ \eqref{1E_k+1_1_3} and %
\eqref{1E_k+1_1_5}, we find
\begin{equation}  \label{1E_k+1_1}
J_{1}\lesssim \delta\left(\norm{\nabla^{k+1}\varrho}_{L^2}^2+%
\norm{\nabla^{k+1}u}_{L^2}^2+\norm{\nabla^{k+2}u}_{L^2}^2\right).
\end{equation}

Next, for the term $J_2$, we utilize the commutator notation \eqref{1commuta}
to rewrite it as
\begin{equation}  \label{1E_k+1_2_0}
\begin{split}
J_2&=-\gamma\int_{\mathbb{R}^3}\nabla^{k+1}
(u\cdot\nabla\varrho)\nabla^{k+1}\varrho\,dx \\
&=-\gamma\int_{\mathbb{R}^3} \left(u\cdot\nabla
\nabla^{k+1}\varrho+[\nabla^{k+1},u]\cdot\nabla\varrho\right)\nabla^{k+1}%
\varrho\,dx \\
&:=J_{21}+J_{22}.
\end{split}%
\end{equation}
By integrating by part, we have
\begin{equation}  \label{1E_k+1_2_1}
\begin{split}
J_{21}&=-\gamma\int_{\mathbb{R}^3} u\cdot\nabla \frac{|\nabla^{k+1}%
\varrho|^2 }{2}\,dx =\frac{\gamma}{2}\int_{\mathbb{R}^3} \mathrm{div}u\, {%
|\nabla^{k+1}\varrho|^2 }\,dx \\
&\lesssim\norm{\nabla u}_{L^\infty}\norm{\nabla^{k+1}\varrho}_{L^2}^2
\lesssim\delta\norm{\nabla^{k+1}\varrho}_{L^2}^2.
\end{split}%
\end{equation}
We use the commutator estimate of Lemma \ref{1commutator} to bound
\begin{equation}  \label{1E_k+1_2_2}
\begin{split}
J_{22} &\lesssim \left(\norm{\nabla u}_{L^\infty}\norm{\nabla^{k}\nabla%
\varrho}_{L^2}+\norm{\nabla^{k+1}u }_{L^2}\norm{\nabla \varrho}%
_{L^\infty}\right)\norm{\nabla^{k+1}\varrho}_{L^2} \\
&\lesssim \delta \left(\norm{\nabla^{k+1}\varrho}_{L^2}^2+%
\norm{\nabla^{k+1}u}_{L^2}^2\right).
\end{split}%
\end{equation}
In light of \eqref{1E_k+1_2_1} and \eqref{1E_k+1_2_2}, we find
\begin{equation}  \label{1E_k+1_2}
J_{2}\lesssim \delta\left(\norm{\nabla^{k+1}\varrho}_{L^2}^2+%
\norm{\nabla^{k+1}u}_{L^2}^2\right).
\end{equation}

Now we estimate the term $J_3$. By H\"older's inequality and Lemma \ref%
{1interpolation}, we have
\begin{equation}  \label{1E_k+1_3}
\begin{split}
J_3&=\int_{\mathbb{R}^3} \nabla^{k}\left(u\cdot\nabla u
\right)\cdot\nabla^{k+2} u\,dx \\
&=\int_{\mathbb{R}^3} \sum_{0\le \ell\le k}C_k^\ell
\left(\nabla^{k-\ell}u\cdot\nabla\nabla^{\ell} u \right) \cdot\nabla^{k+2}
u\,dx \\
&\lesssim \sum_{0\le \ell\le k}\norm{\nabla^{k-\ell} u}_{L^\infty}\norm{
\nabla^{\ell+1} u }_{L^2}\norm{\nabla^{k+2} u}_{L^2} \\
&\lesssim \sum_{0\le \ell\le k} \norm{ \nabla^\alpha u }_{L^2}^{\frac{\ell+1%
}{k+2}}\norm{ \nabla^{k+2} u }_{L^2}^{1-\frac{\ell+1}{k+2}} \norm{ u }%
_{L^2}^{1-\frac{\ell+1}{k+2}}\norm{ \nabla^{k+2} u }_{L^2}^\frac{\ell+1}{k+2}%
\norm{\nabla^{k+2} u}_{L^2} \\
&\lesssim \delta \norm{\nabla^{k+2} u}_{L^2}^2 ,
\end{split}%
\end{equation}
where we have denoted $\alpha$ by
\begin{equation}
\begin{split}
&\frac{k-\ell}{3}=\left(\frac{\alpha}{3}-\frac{1}{2}\right)\times\frac{\ell+1%
}{k+2} +\left(\frac{k+2}{3}-\frac{1}{2}\right)\times\left(1-\frac{\ell+1}{k+2%
}\right) \\
&\quad\Longrightarrow \alpha=\frac{ k+2 }{2(\ell+1) } \le \frac{ k+2 }{2}\le %
\left[\frac{N}{2}\right]+2.
\end{split}%
\end{equation}

Next, we estimate the term $J_4$, and we do the splitting
\begin{equation}  \label{1E_k+1_4_0}
\begin{split}
J_4&=\int_{\mathbb{R}^3} \nabla^{k}\left( h(\varrho)(\bar{\mu}\Delta u+(\bar{%
\mu}+\bar{\lambda})\nabla\mathrm{div} u)\right)\cdot\nabla^{k+2} u\,dx \\
&\approx \int_{\mathbb{R}^3} \nabla^{k}\left(
h(\varrho)\nabla^2u\right)\cdot\nabla^{k+2} u\,dx \\
&=\int_{\mathbb{R}^3} \Big(h(\varrho)\nabla^{k+2}u+C_k^1\nabla(h(\varrho))%
\nabla^{k+1}u+C_k^2\nabla^2(h(\varrho))\nabla^{k}u \\
&\qquad\quad\, +\sum_{3\le \ell\le
k}C_k^\ell\nabla^\ell(h(\varrho))\nabla^{k-\ell+2}u\Big)\cdot\nabla^{k+2}
u\,dx \\
&:=J_{41}+J_{42}+J_{43}+J_{44}.
\end{split}%
\end{equation}
The first three terms can be easily bounded by
\begin{equation}  \label{1E_k+1_4_1}
J_{41}\lesssim \norm{h(\varrho)}_{L^\infty}\norm{\nabla^{k+2}u}_{L^2}^2
\lesssim \delta\norm{\nabla^{k+2}u}_{L^2}^2
\end{equation}
and
\begin{equation}  \label{1E_k+1_4_2}
J_{42}\lesssim\norm{h'(\varrho)}_{L^\infty}\norm{\nabla\varrho}_{L^3}%
\norm{\nabla^{k+1}u}_{L^6}\norm{\nabla^{k+2}u}_{L^2} \lesssim \delta%
\norm{\nabla^{k+2}u}_{L^2}^2,
\end{equation}
and
\begin{equation}  \label{1E_k+1_4_3}
\begin{split}
J_{43}&\lesssim \left(\norm{h''(\varrho)}_{L^\infty}\norm{\nabla\varrho}%
_{L^\infty}\norm{\nabla\varrho}_{L^3}+\norm{h'(\varrho)}_{L^\infty}%
\norm{\nabla^2\varrho}_{L^3}\right)\norm{\nabla^{k}u}_{L^6}%
\norm{\nabla^{k+2}u}_{L^2} \\
&\lesssim \delta\left(\norm{\nabla^{k+1}u}_{L^2}^2+\norm{\nabla^{k+2}u}%
_{L^2}^2\right).
\end{split}%
\end{equation}
The last term $J_{44}$ is much more complicated. We shall split it further
as follows.
\begin{equation}  \label{1E_k+1_4_4}
\begin{split}
J_{44} &\lesssim \sum_{3\le \ell\le k}\int_{\mathbb{R}^3}|\nabla^\ell(h(%
\varrho))||\nabla^{k-\ell+2}u|| \nabla^{k+2} u |\,dx \\
&=\sum_{3\le \ell\le k}\int_{\mathbb{R}^3}|\nabla^{\ell-1}(h^{\prime}(%
\varrho)\nabla\varrho)||\nabla^{k-\ell+2}u|| \nabla^{k+2}u |\,dx \\
&\le \sum_{3\le \ell\le k}\int_{\mathbb{R}^3}\left(\left|h^{\prime}(\varrho)%
\nabla^{\ell}\varrho\right|+\sum_{1\le m\le
\ell-1}\left|\nabla^{m}(h^{\prime}(\varrho))\nabla^{\ell-m}\varrho\right|%
\right)|\nabla^{k-\ell+2}u||\nabla^{k+2} u |\,dx \\
&:=J_{441}+J_{442}.
\end{split}%
\end{equation}
Since $k-\ell+2\le k-1$, we may use H\"older's inequality and Lemma \ref%
{1interpolation} to bound
\begin{equation}  \label{1E_k+1_4_5}
\begin{split}
J_{441}&\lesssim\sum_{3\le \ell\le k}\norm{\nabla^{\ell}\varrho}_{L^2}%
\norm{\nabla^{k-\ell+2}u}_{L^\infty}\norm{ \nabla^{k+2} u}_{L^2} \\
&\lesssim\sum_{3\le \ell\le k}\norm{\varrho}_{L^2}^{1-\frac{\ell}{k+1}} %
\norm{\nabla^{k+1}\varrho}_{L^2}^{\frac{\ell}{k+1}}\norm{\nabla^\alpha u}%
_{L^2}^{\frac{\ell}{k+1}}\norm{\nabla^{k+2}u}_{L^2}^{1-\frac{\ell}{k+1}}%
\norm{ \nabla^{k+2} u}_{L^2} \\
&\lesssim \delta \left(\norm{\nabla^{k+1}\varrho}_{L^2}^2+%
\norm{\nabla^{k+2}u}_{L^2}^2\right),
\end{split}%
\end{equation}
where we have denoted $\alpha$ by
\begin{equation}
\begin{split}
&\frac{{k-\ell+2}}{3} =\left(\frac{\alpha}{3}-\frac{1}{2}\right)\times \frac{%
\ell}{k+1}+\left(\frac{{k+2}}{3}-\frac{1}{2}\right)\times\left(1-\frac{\ell}{%
k+1}\right) \\
&\quad \Longrightarrow \alpha=\frac{3(k+1)}{2\ell}+1 \le \frac{ k+3 }{2}\le %
\left[\frac{N}{2}\right]+2\, \text{ since }\ell\ge 3.
\end{split}%
\end{equation}

For the term $J_{442}$, noting that $1\le m\le \ell-1\le k-1$, we may then
use H\"older's inequality, Lemma \ref{1composition} and Lemma \ref%
{1interpolation} to bound
\begin{equation}  \label{1E_k+1_4_6}
\begin{split}
J_{442} &\lesssim \sum_{3\le \ell\le k} \sum_{1\le m\le \ell-1}%
\norm{\nabla^{m}(h'(\varrho))}_{L^\infty}\norm{\nabla^{\ell-m}\varrho}%
_{L^\infty}\norm{\nabla^{k-\ell+2}u}_{L^2}\norm{\nabla^{k+2} u }_{L^2} \\
&\lesssim\sum_{3\le \ell\le k} \sum_{1\le m\le \ell-1}\norm{\nabla^m\varrho}%
_{L^2} ^{1/4}\norm{\nabla^{m+2}\varrho}_{L^2}^{3/4}\norm{\nabla^{\ell-m}%
\varrho}_{L^2} ^{1/4}\norm{\nabla^{\ell-m+2}\varrho}_{L^2}^{3/4} \\
&\qquad\qquad\qquad\quad\ \times\norm{\nabla^{k-\ell+2}u}_{L^2}\norm{
\nabla^{k+2} u }_{L^2}.
\end{split}%
\end{equation}
To estimate the right hand side of \eqref{1E_k+1_4_6}, we divide it into two
cases.\smallskip\smallskip

\textbf{Case 1}: For $m=1$ or $m=\ell-1$, we have
\begin{equation}  \label{1E_k+1_4_61}
\begin{split}
& \norm{\nabla\varrho}_{L^2} ^{1/4}\norm{\nabla^{3}\varrho}_{L^2} ^{3/4}%
\norm{\nabla^{\ell-1}\varrho}_{L^2} ^{1/4}\norm{\nabla^{\ell +1}\varrho}%
_{L^2} ^{{3}/{4}} \\
&\quad\lesssim \delta \norm{ \varrho}_{L^2}^{\frac{1}{4}\left(1-\frac{\ell-1%
}{k+1}\right)}\norm{\nabla^{k+1}\varrho}_{L^2}^{\frac{1}{4}\left( \frac{%
\ell-1}{k+1}\right)} \norm{ \varrho}_{L^2} ^{\frac{3}{4}\left(1-\frac{\ell+1%
}{k+1}\right)}\norm{\nabla^{k+1}\varrho}_{L^2}^{\frac{3}{4}\left( \frac{%
\ell+1}{k+1}\right)} \\
&\quad \lesssim\delta^{2-\frac{2\ell+1}{2(k+1)}}\norm{\nabla^{k+1}\varrho}%
_{L^2} ^{ \frac{2\ell+1}{2(k+1)}}.
\end{split}%
\end{equation}
Hence, by \eqref{1E_k+1_4_61} and using Lemma \ref{1interpolation}, we have
\begin{equation}  \label{1E_k+1_4_62}
\begin{split}
&\norm{\nabla\varrho}_{L^2}^{1/4}\norm{\nabla^{3}\varrho}_{L^2} ^{3/4}%
\norm{\nabla^{\ell-1}\varrho}_{L^2}^{1/4}\norm{\nabla^{\ell +1}\varrho}%
_{L^2}^{{3}/{4}}\norm{\nabla^{k-\ell+2}u}_{L^2}\norm{\nabla^{k+2} u}_{L^2} \\
&\quad\lesssim\delta^{2-\frac{2\ell+1}{2(k+1)}}\norm{\nabla^{k+1}\varrho}%
_{L^2} ^{ \frac{2\ell+1}{2(k+1)}}\norm{\nabla^\alpha u}_{L^2}^{\frac{2\ell+1%
}{2(k+1)}}\norm{\nabla^{k+1}u}_{L^2}^{1- \frac{2\ell+1}{2(k+1)}}%
\norm{\nabla^{k+2}u}_{L^2} \\
&\quad\lesssim \delta^2\left(\norm{\nabla^{k+1}\varrho}_{L^2}^2+%
\norm{\nabla^{k+2}u}_{L^2}^2\right),
\end{split}%
\end{equation}
where we have denoted $\alpha$ by
\begin{equation}  \label{tttt}
\begin{split}
&{k-\ell+2} =\alpha\times \frac{2\ell+1}{2(k+1)}+({k+1})\times\left(1-{\frac{%
2\ell+1}{2(k+1)}}\right) \\
&\quad \Longrightarrow \alpha=\frac{3(k+1)}{2\ell+1}\le \frac{3(k+1)}{7}\le %
\left[ \frac{N}{2}\right]+1\,\text{ since }\ell\ge 3.
\end{split}%
\end{equation}

\textbf{Case 2}: For $2\le m\le \ell-2$, noting also that $\ell-m\ge 2$,
then we bound
\begin{equation}  \label{1E_k+1_4_63}
\begin{split}
&\norm{\nabla^m\varrho}_{L^2} ^{1/4}\norm{\nabla^{m+2}\varrho}_{L^2}^{3/4}%
\norm{\nabla^{\ell-m}\varrho}_{L^2} ^{1/4}\norm{\nabla^{\ell-m+2}\varrho}%
_{L^2} ^{3/4} \\
&\quad\lesssim\left(\norm{\nabla^2\varrho}_{L^2}^{1-\frac{m-2}{k-1}}%
\norm{\nabla^{k+1}\varrho}_{L^2}^{\frac{m-2}{k-1}}\right)^{1/4}\left(%
\norm{\nabla^2\varrho}_{L^2}^{1-\frac{m}{k-1}}\norm{\nabla^{k+1}\varrho}%
_{L^2}^{\frac{m}{k-1}}\right) ^{3/4} \\
&\quad\quad\ \times\left(\norm{\nabla^2\varrho}_{L^2}^{1-\frac{\ell-m-2}{k-1}%
}\norm{\nabla^{k+1}\varrho}_{L^2}^{\frac{\ell-m-2}{k-1}}\right)^{1/4}\left(%
\norm{\nabla^2\varrho}_{L^2}^{1-\frac{\ell-m}{k-1}}\norm{\nabla^{k+1}\varrho}%
_{L^2}^{\frac{\ell-m}{k-1}}\right) ^{3/4} \\
&\quad\lesssim\norm{\nabla^2\varrho}_{L^2}^{2-\frac{\ell-1}{k-1}}%
\norm{\nabla^{k+1}\varrho}_{L^2}^{\frac{\ell-1}{k-1}}.
\end{split}%
\end{equation}
Hence, by \eqref{1E_k+1_4_63} and using Lemma \ref{1interpolation} again, we
have
\begin{equation}  \label{1E_k+1_4_65}
\begin{split}
&\norm{\nabla^m\varrho}_{L^2} ^{1/4}\norm{\nabla^{m+2}\varrho}_{L^2}^{3/4}%
\norm{\nabla^{\ell-m}\varrho}_{L^2} ^{1/4}\norm{\nabla^{\ell-m+2}\varrho}%
_{L^2}^{3/4}\norm{\nabla^{k-\ell+2}u}_{L^2}\norm{ \nabla^{k+2} u }_{L^2} \\
&\quad\lesssim \delta^{2-\frac{\ell-1}{k-1}} \norm{\nabla^{k+1}\varrho}%
_{L^2}^{\frac{\ell-1}{k-1}}\norm{\nabla^\alpha u}_{L^2}^{ \frac{\ell-1}{k-1}}%
\norm{\nabla^{k+1}u}_{L^2}^{1-\frac{\ell-1}{k-1}}\norm{\nabla^{k+2}u}_{L^2}
\\
&\quad\lesssim \delta^2 \left(\norm{\nabla^{k+1}\varrho}_{L^2}^2+%
\norm{\nabla^{k+1}u}_{L^2}^2+\norm{\nabla^{k+2}u}_{L^2}^2\right),
\end{split}%
\end{equation}
where we have denoted $\alpha$ by
\begin{equation}
\begin{split}
&{k-\ell+2} =\alpha\times \frac{\ell-1}{k-1}+({k+1})\times\left(1-\frac{%
\ell-1}{k-1}\right) \\
&\quad \Longrightarrow \alpha=2.
\end{split}%
\end{equation}

Therefore, we deduce from the two cases above that
\begin{equation}
J_{442}\lesssim \delta^2 \left(\norm{\nabla^{k+1}\varrho}_{L^2}^2+%
\norm{\nabla^{k+1}u}_{L^2}^2+\norm{\nabla^{k+2}u}_{L^2}^2\right),
\end{equation}
and this together with \eqref{1E_k+1_4_5} implies
\begin{equation}  \label{1E_k+1_4_7}
J_{44}\lesssim \delta \left(\norm{\nabla^{k+1}\varrho}_{L^2}^2+%
\norm{\nabla^{k+1}u}_{L^2}^2+\norm{\nabla^{k+2}u}_{L^2}^2\right).
\end{equation}
By the estimates \eqref{1E_k+1_4_1}, \eqref{1E_k+1_4_2}, \eqref{1E_k+1_4_3}
and \eqref{1E_k+1_4_7}, we obtain
\begin{equation}  \label{1E_k+1_4}
J_{4}\lesssim \delta \left(\norm{\nabla^{k+1}\varrho}_{L^2}^2+%
\norm{\nabla^{k+1}u}_{L^2}^2+\norm{\nabla^{k+2}u}_{L^2}^2\right).
\end{equation}

Finally, it remains to estimate the last term $J_{5}$. To begin with, we
split
\begin{equation}
\begin{split}
J_{5}& =\int_{\mathbb{R}^{3}}\nabla ^{k}\left( f(\varrho )\nabla \varrho
\right) \cdot \nabla ^{k+2}u\,dx \\
& =\int_{\mathbb{R}^{3}}\Big(f(\varrho )\nabla ^{k+1}\varrho
+C_{k}^{1}\nabla (f(\varrho ))\nabla ^{k}\varrho +\sum_{2\leq \ell \leq
k}C_{k}^{\ell }\nabla ^{\ell }(f(\varrho ))\nabla ^{k-\ell +1}\varrho \Big)%
\cdot \nabla ^{k+2}u\,dx \\
& :=J_{51}+J_{52}+J_{53}.
\end{split}
\label{1E_k+1_5_0}
\end{equation}%
The first two terms can be easily bounded by
\begin{equation}
J_{51}\lesssim \norm{f(\varrho)}_{L^{\infty }}\norm{\nabla^{k+1}\varrho}%
_{L^{2}}\norm{\nabla^{k+2}u}_{L^{2}}\lesssim \delta \left( %
\norm{\nabla^{k+1}\varrho}_{L^{2}}^{2}+\norm{\nabla^{k+2}u}%
_{L^{2}}^{2}\right)   \label{1E_k+1_5_1}
\end{equation}%
and
\begin{equation}
J_{52}\lesssim \norm{\nabla\varrho}_{L^{3}}\norm{\nabla^{k}\varrho}_{L^{6}}%
\norm{\nabla^{k+2}u}_{L^{2}}\lesssim \delta \left( \norm{\nabla^{k+1}\varrho}%
_{L^{2}}^{2}+\norm{\nabla^{k+2}u}_{L^{2}}^{2}\right) .  \label{1E_k+1_5_2}
\end{equation}

We now focus on the most delicate term $J_{53}$. We shall split it further
as follows.
\begin{equation}  \label{1E_k+1_5_3}
\begin{split}
J_{53} &\lesssim \sum_{2\le \ell\le k}\int_{\mathbb{R}^3}|\nabla^\ell(f(%
\varrho))||\nabla^{k-\ell+1}\varrho|| \nabla^{k+2} u |\,dx \\
&=\sum_{2\le \ell\le k}\int_{\mathbb{R}^3}|\nabla^{\ell-1}(f^{\prime}(%
\varrho)\nabla\varrho)||\nabla^{k-\ell+1}\varrho||\nabla^{k+2} u |\,dx \\
&\le \int_{\mathbb{R}^3} \left(|f^{\prime}(\varrho)\nabla^2\varrho|+
|f^{\prime\prime}(\varrho)\nabla\varrho \nabla
\varrho|\right)|\nabla^{k-1}\varrho|| \nabla^{k+2} u |\,dx \\
&\quad + \sum_{3\le \ell\le k}\int_{\mathbb{R}^3}\left(|f^{\prime}(\varrho)%
\nabla^{\ell}\varrho|+\sum_{1\le m\le
\ell-1}|\nabla^{m}(f^{\prime}(\varrho))\nabla^{\ell-m}\varrho|\right)|%
\nabla^{k-\ell+1}\varrho||\nabla^{k+2} u |\,dx \\
&:=J_{531}+J_{532}+J_{533}.
\end{split}%
\end{equation}
By H\"older's inequality, Lemma \ref{1interpolation}, we estimate $J_{531}$
by
\begin{equation}  \label{1E_k+1_5_4}
\begin{split}
J_{531}&\lesssim\left(\norm{\nabla^2\varrho}_{L^2}+\norm{\nabla\varrho}_{L^3}%
\norm{\nabla\varrho}_{L^6}\right)\norm{\nabla^{k-1}\varrho}_{L^\infty}\norm{
\nabla^{k+2}u}_{L^2} \\
&\lesssim\norm{\nabla^2\varrho}_{L^2}\norm{\varrho}_{L^2}^{\frac{\frac{1}{2}%
}{k+1}}\norm{\nabla^{k+1}\varrho}_{L^2}^{\frac{k+\frac{1}{2}}{k+1}}\norm{
\nabla^{k+2}u}_{L^2} \\
&\lesssim\norm{\nabla^2\varrho}_{L^2}^\frac{3}{4}\norm{\varrho}_{L^2}^{\frac{%
1}{4}-\frac{\frac{1}{2}}{k+1}}\norm{\nabla^{k+1}\varrho}_{L^2}^{\frac{\frac{1%
}{2}}{k+1}}\norm{\varrho}_{L^2}^{\frac{\frac{1}{2}}{k+1}}\norm{\nabla^{k+1}%
\varrho}_{L^2}^{\frac{k+\frac{1}{2}}{k+1}}\norm{ \nabla^{k+2}u}_{L^2} \\
&\lesssim\delta\left(\norm{\nabla^{k+1}\varrho}_{L^2}^2+\norm{\nabla^{k+2}u}%
_{L^2}^2\right).
\end{split}%
\end{equation}
Since $\ell\ge 3$, $k-\ell+1\le k-2$, we may use H\"older's inequality and
Lemma \ref{1interpolation} to bound $J_{532}$ by
\begin{equation}  \label{1E_k+1_5_5}
\begin{split}
J_{532} &\lesssim\norm{\nabla^{\ell}\varrho}_{L^2}\norm{\nabla^{k-\ell+1}%
\varrho}_{L^\infty}\norm{\nabla^{k+2} u}_{L^2} \\
&\lesssim\norm{\varrho}_{L^2}^{1-\frac{\ell}{k+1}}\norm{\nabla^{k+1}\varrho}%
_{L^2}^{\frac{\ell}{k+1}}\norm{\nabla^\alpha \varrho}_{L^2}^{\frac{\ell}{k+1}%
}\norm{\nabla^{k+1}\varrho}_{L^2}^{1-\frac{\ell}{k+1}}\norm{ \nabla^{k+2} u}%
_{L^2} \\
&\lesssim \delta \left(\norm{\nabla^{k+1}\varrho}_{L^2}^2+%
\norm{\nabla^{k+2}u}_{L^2}^2\right),
\end{split}%
\end{equation}
where we have denoted $\alpha$ by
\begin{equation}
\begin{split}
&\frac{{k-\ell+1}}{3} =\left(\frac{\alpha}{3}-\frac{1}{2}\right)\times\frac{%
\ell}{k+1}+\left(\frac{{k+1}}{3}-\frac{1}{2}\right)\times(1-\frac{\ell}{k+1})
\\
&\quad \Longrightarrow \alpha=\frac{3(k+1)}{2\ell}\le \frac{k+1}{2} \le %
\left[\frac{N}{2}\right]+1\,\text{ since }\ell\ge 3.
\end{split}%
\end{equation}

For the term $J_{533}$, noticing that $1\le m\le \ell-1\le k-1$, we may then
use H\"older's inequality, Lemma \ref{1composition} and Lemma \ref%
{1interpolation} to bound
\begin{equation}  \label{1E_k+1_5_6}
\begin{split}
J_{533} &\lesssim \sum_{3\le \ell\le k} \sum_{1\le m\le \ell-1}%
\norm{\nabla^{m}(h'(\varrho))}_{L^\infty}\norm{\nabla^{\ell-m}\varrho}%
_{L^\infty}\norm{\nabla^{k-\ell+1}\varrho}_{L^2}\norm{\nabla^{k+2} u }_{L^2}
\\
&\lesssim\sum_{3\le \ell\le k} \sum_{1\le m\le \ell-1}\norm{\nabla^m\varrho}%
_{L^2} ^{1/4}\norm{\nabla^{m+2}\varrho}_{L^2}^{3/4}\norm{\nabla^{\ell-m}%
\varrho}_{L^2} ^{1/4}\norm{\nabla^{\ell-m+2}\varrho}_{L^2}^{3/4} \\
&\qquad\qquad\qquad\quad\ \times\norm{\nabla^{k-\ell+1}\varrho}_{L^2}\norm{
\nabla^{k+2} u }_{L^2}.
\end{split}%
\end{equation}
To estimate the right hand side of \eqref{1E_k+1_5_6}, we divide it into two
cases.\smallskip\smallskip

\textbf{Case 1}: For $m=1$ or $m=\ell-1$, similarly as in \eqref{1E_k+1_4_61}%
--\eqref{tttt}, by Lemma \ref{1interpolation} we have
\begin{equation}  \label{1E_k+1_5_7}
\begin{split}
&\norm{\nabla\varrho}_{L^2} ^{1/4}\norm{\nabla^{3}\varrho}_{L^2}^{3/4}%
\norm{\nabla^{\ell-1}\varrho}_{L^2} ^{1/4}\norm{\nabla^{\ell +1}\varrho}%
_{L^2}^{{3}/{4}}\norm{\nabla^{k-\ell+1}\varrho}_{L^2}\norm{\nabla^{k+2} u}%
_{L^2} \\
&\quad\lesssim\delta^{2-\frac{2\ell+1}{2(k+1)}}\norm{\nabla^{k+1}\varrho}%
_{L^2} ^{ \frac{2\ell+1}{2(k+1)}}\norm{\nabla^\alpha \varrho}_{L^2}^{\frac{%
2\ell+1}{2(k+1)}}\norm{\nabla^{k+1}\varrho}_{L^2}^{1-\frac{2\ell+1}{2(k+1)}}%
\norm{\nabla^{k+2}u}_{L^2} \\
&\quad\lesssim \delta^2 \left(\norm{\nabla^{k+1}\varrho}_{L^2}^2+%
\norm{\nabla^{k+2}u}_{L^2}^2\right),
\end{split}%
\end{equation}
where we have denoted $\alpha$ by
\begin{equation}
\begin{split}
&{k-\ell+1} =\alpha\times \frac{2\ell+1}{2(k+1)}+({k+1})\times\left(1- \frac{%
2\ell+1}{2(k+1)}\right) \\
&\quad\Longrightarrow \alpha=\frac{k+1}{2\ell+1}\le \frac{k+1}{7}\le \left[
\frac{N}{7}\right]+1\,\text{ since }\ell\ge 3.
\end{split}%
\end{equation}

\textbf{Case 2}: For $2\le m\le \ell-2$, noting also that $\ell-m\ge 2$, as
in \eqref{1E_k+1_4_63}, we may bound
\begin{equation}  \label{1E_k+1_5_8}
\norm{\nabla^m\varrho}_{L^2}^{1/4}\norm{\nabla^{m+2}\varrho}_{L^2} ^{3/4}%
\norm{\nabla^{\ell-m}\varrho}_{L^2}^{1/4}\norm{\nabla^{\ell-m+2}\varrho}%
_{L^2} ^{3/4} \lesssim \norm{\nabla^2\varrho}^{2-\frac{\ell-1}{k-1}} %
\norm{\nabla^{k+1}\varrho}_{L^2}^{\frac{\ell-1}{k-1}}.
\end{equation}
Notice that we may not simply let $\norm{\nabla^2\varrho}\lesssim \delta$ as
in \eqref{1E_k+1_4_65} since it would be out of reach for some $\ell$. We
will adjust the index as follows, by Lemma \ref{1interpolation},
\begin{equation}  \label{1E_k+1_5_9}
\begin{split}
\norm{\nabla^2\varrho}^{2-\frac{\ell-1}{k-1}}\norm{\nabla^{k+1}\varrho}%
_{L^2}^{\frac{\ell-1}{k-1}} &\lesssim \delta^{2-\frac{\ell-1}{k-1}-\frac{k+1%
}{2(k-1)}}\norm{\nabla^2\varrho}_{L^2}^\frac{k+1}{2(k-1)} %
\norm{\nabla^{k+1}\varrho}_{L^2}^{\frac{\ell-1}{k-1}} \\
&\lesssim \delta^{2-\frac{\ell-1}{k-1}-\frac{k+1}{2(k-1)}}\left(%
\norm{\varrho}_{L^2}^{1-\frac{2}{k+1}} \norm{\nabla^{k+1}\varrho}_{L^2}^{%
\frac{2}{k+1}}\right)^\frac{k+1}{2(k-1)}\norm{\nabla^{k+1}\varrho}_{L^2}^{%
\frac{\ell-1}{k-1}} \\
&\lesssim \delta^{2-\frac{\ell}{k-1}}\norm{\nabla^{k+1}\varrho}_{L^2}^{\frac{%
\ell}{k-1}}.
\end{split}%
\end{equation}
Hence, by \eqref{1E_k+1_5_8}--\eqref{1E_k+1_5_9} and using Lemma \ref%
{1interpolation} again, we have
\begin{equation}
\begin{split}
&\norm{\nabla^m\varrho}_{L^2} ^{1/4}\norm{\nabla^{m+2}\varrho}_{L^2}^{3/4}%
\norm{\nabla^{\ell-m}\varrho}_{L^2} ^{1/4}\norm{\nabla^{\ell-m+2}\varrho}%
_{L^2}^{3/4}\norm{\nabla^{k-\ell+1}\varrho}_{L^2}\norm{ \nabla^{k+2} u}_{L^2}
\\
&\quad\lesssim\delta^{2-\frac{\ell}{k-1}} \norm{\nabla^{k+1}\varrho}_{L^2}^{%
\frac{\ell}{k-1}}\norm{\nabla^\alpha\varrho}_{L^2}^{\frac{\ell}{k-1}}%
\norm{\nabla^{k+1}\varrho}_{L^2}^{1-\frac{\ell}{k-1}}\norm{\nabla^{k+2} u}%
_{L^2} \\
&\quad\lesssim\delta^2 \left(\norm{\nabla^{k+1}\varrho}_{L^2}^2+%
\norm{\nabla^{k+2}u}_{L^2}^2\right),
\end{split}%
\end{equation}
where we have denoted $\alpha$ by
\begin{equation}
\begin{split}
&{k-\ell+1} =\alpha\times \frac{\ell}{k-1}+({k+1})\times\left(1-\frac{\ell}{%
k-1}\right) \\
&\quad \Longrightarrow \alpha=2.
\end{split}%
\end{equation}

Therefore, we deduce from the two cases above that
\begin{equation}
J_{533}\lesssim \delta ^{2}\left( \norm{\nabla^{k+1}\varrho}_{L^{2}}^{2}+%
\norm{\nabla^{k+2}u}_{L^{2}}^{2}\right) ,  \label{1E_k+1_5_10}
\end{equation}%
and this together with \eqref{1E_k+1_5_4} and \eqref{1E_k+1_5_5} implies
\begin{equation}
J_{53}\lesssim \delta \left( \norm{\nabla^{k+1}\varrho}_{L^{2}}^{2}+%
\norm{\nabla^{k+2}u}_{L^{2}}^{2}\right) .  \label{1E_k+1_5_11}
\end{equation}%
By the estimates \eqref{1E_k+1_5_1}, \eqref{1E_k+1_5_2} and %
\eqref{1E_k+1_5_11}, we obtain
\begin{equation}
J_{5}\lesssim \delta \left( \norm{\nabla^{k+1}\varrho}_{L^{2}}^{2}+%
\norm{\nabla^{k+1}u}_{L^{2}}^{2}+\norm{\nabla^{k+2}u}_{L^{2}}^{2}\right) .
\label{1E_k+1_5}
\end{equation}

Summing up the estimates for $J_{1}\sim J_{5}$, $i.e.$, \eqref{1E_k+1_1}, %
\eqref{1E_k+1_2}, \eqref{1E_k+1_3}, \eqref{1E_k+1_4} and \eqref{1E_k+1_5},
by \eqref{1coercive}, we deduce \eqref{1E_k+1} for $0\leq k\leq N-1$.
\end{proof}

Now, we will use the equations \eqref{1NS2} to recover the dissipation
estimate for $\varrho$.

\begin{lemma}
\label{1EkEk+1le} Assume that $0\le k\le N-1$, then we have
\begin{equation}  \label{1E_kE_k+1}
\frac{d}{dt}\int_{\mathbb{R}^3} \nabla^ku\cdot\nabla\nabla^k\varrho\,dx +C%
\norm{\nabla^{k+1}\varrho}_{L^2}^2 \lesssim\norm{\nabla^{k+1}u}_{L^2}^2+%
\norm{\nabla^{k+2}u}_{L^2}^2.
\end{equation}
\end{lemma}

\begin{proof}
Let $0\le k\le N-1$. Applying $\nabla^k$ to $\eqref{1NS2}_2$ and then
multiplying by $\nabla\nabla^k\varrho$, we obtain
\begin{equation}  \label{1E_kE_k+1_0}
\begin{split}
&\gamma\bar{\rho} \int_{\mathbb{R}^3} |\nabla^{k+1}\varrho|^2\,dx \le -\int_{%
\mathbb{R}^3} \nabla^k \partial_tu \cdot\nabla\nabla^k\varrho\,dx+C%
\norm{\nabla^{k+2} u}_{L^2} \norm{\nabla^{k+1} \varrho}_{L^2} \\
&\qquad\qquad\qquad\ +\norm{\nabla^{k}\left(u\cdot\nabla
u+h(\varrho)(\bar{\mu}\Delta u+(\bar{\mu}+\bar{\lambda})\nabla{\rm div}
u)+f(\varrho)\nabla\varrho\right)}_{L^2}\norm{\nabla^{k+1} \varrho}_{L^2}.
\end{split}%
\end{equation}

The delicate first term in the right hand side of \eqref{1E_kE_k+1_0}
involves $\nabla^{k}\partial_t u$, and the key idea is to integrate by parts
in the $t$-variable and use the continuity equation. Thus integrating by
parts for both the $t$- and $x$-variables, we obtain
\begin{equation}  \label{1E_kE_k+1_1}
\begin{split}
&-\int_{\mathbb{R}^3} \nabla^{k} \partial_tu\cdot\nabla\nabla^k\varrho\,dx \\
&\quad=-\frac{d}{dt}\int_{\mathbb{R}^3}
\nabla^{k}u\cdot\nabla\nabla^k\varrho\,dx-\int_{\mathbb{R}^3} \nabla^{k}
\mathrm{div}u\cdot\nabla^{k}\partial_t\varrho\,dx \\
&\quad=-\frac{d}{dt}\int_{\mathbb{R}^3}
\nabla^{k}u\cdot\nabla\nabla^k\varrho\,dx+\bar{\rho}\norm{\nabla^{k} {\rm
div}u}_{L^2}^2+\int_{\mathbb{R}^3}\nabla^{k} \mathrm{div}u\cdot\nabla^{k}(%
\varrho\mathrm{div}u+u\cdot\nabla\varrho)\,dx.
\end{split}%
\end{equation}
For $k=0$, we easily bound the last term in \eqref{1E_kE_k+1_1} by
\begin{equation}  \label{1E_kE_k+1_2}
-\int_{\mathbb{R}^3}\mathrm{div}u\cdot (\varrho\mathrm{div}%
u+u\cdot\nabla\varrho)\,dx \lesssim \delta\left(\norm{\nabla \varrho}%
_{L^2}^2+\norm{\nabla u}_{L^2}^2\right).
\end{equation}
For $1\le k\le N-1$, we integrate by parts to have
\begin{equation}  \label{1E_kE_k+1_3}
\int_{\mathbb{R}^3} \nabla^{k} \mathrm{div}u\cdot\nabla^{k}(\varrho\mathrm{%
div}u+u\cdot\nabla\varrho)\,dx=-\int_{\mathbb{R}^3} \nabla^{k+1} \mathrm{div}%
u\cdot\nabla^{k-1}(\varrho\mathrm{div}u+u\cdot\nabla\varrho)\,dx.
\end{equation}
Recalling from the derivations of the estimates of $I_1$ and $I_2$ in Lemma %
\ref{1Ekle}, we have already proved that
\begin{equation}  \label{1E_kE_k+1_4}
\norm{\nabla^{k-1}\left(\varrho{\rm div}u+u\cdot\nabla\varrho\right)}%
_{L^2}\lesssim \delta\left(\norm{\nabla^{k+1}\varrho}_{L^2}+%
\norm{\nabla^{k+1}u}_{L^2}\right).
\end{equation}
Thus, in view of \eqref{1E_kE_k+1_1}--\eqref{1E_kE_k+1_4}, together with
Cauchy's inequality, we obtain
\begin{equation}  \label{1E_kE_k+1_5}
\begin{split}
&-\int_{\mathbb{R}^3} \nabla^{k} u_t\cdot\nabla\nabla^k\varrho\,dx \\
&\quad\le-\frac{d}{dt}\int_{\mathbb{R}^3} \nabla^{k}
u\cdot\nabla\nabla^k\varrho\,dx+C\left(\norm{\nabla^{k+1} u}_{L^2}^2+%
\norm{\nabla^{k+2} u}_{L^2}^2\right)+ \delta \norm{\nabla^{k+1}\varrho}%
_{L^2}^2.
\end{split}%
\end{equation}

On the other hand, recalling the derivations of the estimates of $J_3$, $J_4$
and $J_5$ in Lemma \ref{1Ek+1le}, we have already proved that
\begin{equation}  \label{1E_kE_k+1_6}
\begin{split}
&\norm{\nabla^{k}\left(u\cdot\nabla u+h(\varrho)(\bar{\mu}\Delta
u+(\bar{\mu}+\bar{\lambda})\nabla{\rm div} u)+f(\varrho)\nabla\varrho\right)}%
_{L^2} \\
&\quad\lesssim \delta \left(\norm{\nabla^{k+1}\varrho}_{L^2}+%
\norm{\nabla^{k+1}u}_{L^2}+\norm{\nabla^{k+2}u}_{L^2}\right).
\end{split}%
\end{equation}
Plugging the estimates \eqref{1E_kE_k+1_5}--\eqref{1E_kE_k+1_6} into %
\eqref{1E_kE_k+1_0}, by Cauchy's inequality, since $\delta$ is small, we
then obtain \eqref{1E_kE_k+1}.
\end{proof}


\subsection{Energy evolution of negative Sobolev norms}

\label{sec 3.2} 

In this subsection, we will derive the evolution of the negative Sobolev
norms of the solution. In order to estimate the nonlinear terms, we need to
restrict ourselves to that $s\in (0,3/2)$. We will establish the following
lemma.

\begin{lemma}
\label{1Esle} For $s\in(0,1/2]$, we have
\begin{equation}  \label{1E_s}
\frac{d}{dt}\int_{\mathbb{R}^3} \gamma|\Lambda^{-s} \varrho|^2+|\Lambda^{-s}
u|^2\,dx +C\norm{\nabla\Lambda^{-s} u}_{L^2}^2 \lesssim \norm{
\nabla\varrho,\nabla u}_{H^1}^2 \left(\norm{\Lambda^{-s}\varrho}_{L^2}+%
\norm{\Lambda^{-s}u}_{L^2}\right);
\end{equation}
and for $s\in(1/2,3/2)$, we have
\begin{equation}  \label{1E_s2}
\begin{split}
&\frac{d}{dt}\int_{\mathbb{R}^3} \gamma|\Lambda^{-s}\varrho|^2+|\Lambda^{-s}
u|^2\,dx +C\norm{\nabla\Lambda^{-s} u}_{L^2}^2 \\
&\quad\lesssim \norm{ (\varrho, u)}_{L^2}^{s-1/2}\norm{(\nabla\varrho,
\nabla u)}_{H^1}^{5/2-s}\left(\norm{\Lambda^{-s}\varrho}_{L^2}+%
\norm{\Lambda^{-s}u}_{L^2}\right).
\end{split}%
\end{equation}
\end{lemma}

\begin{proof}
Applying $\Lambda^{-s}$ to $\eqref{1NS2}_1, \eqref{1NS2}_2$ and multiplying
the resulting by $\displaystyle\gamma\Lambda^{-s}\varrho, \Lambda^{-s} u$
respectively, summing up and then integrating over $\mathbb{R}^3$ by parts,
we obtain
\begin{equation}  \label{1E_s_0}
\begin{split}
&\frac{1}{2}\frac{d}{dt}\int_{\mathbb{R}^3} \gamma|\Lambda^{-s} \varrho|^2+|
\Lambda^{-s} u|^2\,dx +\int_{\mathbb{R}^3}\bar{\mu}|\nabla \Lambda^{-s}
u|^2+(\bar{\mu}+\bar{\lambda})| \mathrm{div} \Lambda^{-s} u|^2\,dx \\
&\quad=\int_{\mathbb{R}^3}\gamma\Lambda^{-s}(-\varrho\mathrm{div}%
u-u\cdot\nabla\varrho)\Lambda^{-s}\varrho \\
&\qquad\ -\Lambda^{-s}\left(u\cdot\nabla u+h(\varrho)(\bar{\mu}\Delta u+(%
\bar{\mu}+\bar{\lambda})\nabla\mathrm{div}u)+f(\varrho)\nabla\varrho\right)%
\cdot\Lambda^{-s} u\,dx \\
&\quad:=W_1+W_2+W_3+W_4+W_5.
\end{split}%
\end{equation}

We now restrict the value of $s$ in order to estimate the nonlinear terms in
the right hand side of \eqref{1E_s_0}. If $s\in (0,1/2]$, then $1/2+s/3<1$
and $3/s\geq 6$. Then using the estimate \eqref{1Riesz es} of Riesz
potential in Lemma \ref{1Riesz} and the Sobolev interpolation of Lemma \ref%
{1interpolation}, together with H\"{o}lder's and Young's inequalities, we
obtain
\begin{equation}
\begin{split}
W_{1}& =-\gamma \int_{\mathbb{R}^{3}}\Lambda ^{-s}(\varrho \mathrm{div}%
u)\Lambda ^{-s}\varrho \,dx\lesssim \norm{\Lambda^{-s}(\varrho{\rm div}u )}%
_{L^{2}}\norm{\Lambda^{-s}\varrho}_{L^{2}} \\
& \lesssim \norm{\varrho{\rm div}u }_{L^{\frac{1}{1/2+s/3}}}%
\norm{\Lambda^{-s}\varrho}_{L^{2}}\lesssim \norm{\varrho}_{L^{3/s}}%
\norm{\nabla u }_{L^{2}}\norm{\Lambda^{-s}\varrho}_{L^{2}} \\
& \lesssim \norm{ \nabla\varrho}_{L^{2}}^{1/2-s}\norm{ \nabla^2\varrho}%
_{L^{2}}^{1/2+s}\norm{\nabla u}_{L^{2}}\norm{\Lambda^{-s}\varrho}_{L^{2}} \\
& \lesssim \left( \norm{ \nabla\varrho}_{L^{2}}^{2}+\norm{ \nabla^2\varrho}%
_{L^{2}}^{2}+\norm{\nabla u}_{L^{2}}^{2}\right) \norm{\Lambda^{-s}\varrho}%
_{L^{2}}.
\end{split}
\label{1E_s_1}
\end{equation}%
Similarly, we can bound the remaining terms by
\begin{eqnarray}
&&W_{2}=-\gamma \int_{\mathbb{R}^{3}}\Lambda ^{-s}(u\cdot \nabla \varrho
)\Lambda ^{-s}\varrho \,dx\lesssim \left( \norm{ \nabla u}_{L^{2}}^{2}+%
\norm{\nabla^2 u}_{L^{2}}^{2}+\norm{ \nabla\varrho}_{L^{2}}^{2}\right) %
\norm{\Lambda^{-s}\varrho}_{L^{2}},  \label{1E_s_2} \\
&&W_{3}=-\int_{\mathbb{R}^{3}}\Lambda ^{-s}\left( u\cdot \nabla u\right)
\cdot \Lambda ^{-s}u\,dx\lesssim \left( \norm{ \nabla u}_{L^{2}}^{2}+%
\norm{\nabla^2 u}_{L^{2}}^{2}\right) \norm{\Lambda^{-s}u}_{L^{2}},
\label{1E_s_3} \\
&&%
\begin{split}
W_{4}& =-\int_{\mathbb{R}^{3}}\Lambda ^{-s}\left( h(\varrho )(\bar{\mu}%
\Delta u+(\bar{\mu}+\bar{\lambda})\nabla \mathrm{div}u)\right) \Lambda
^{-s}u\,dx \\
& \lesssim \left( \norm{ \nabla\varrho}_{L^{2}}^{2}+\norm{\nabla^2\varrho}%
_{L^{2}}^{2}+\norm{\nabla^2 u }_{L^{2}}^{2}\right) \norm{\Lambda^{-s}u}%
_{L^{2}},
\end{split}
\label{1E_s_4} \\
&&W_{5}=-\int_{\mathbb{R}^{3}}\Lambda ^{-s}\left( f(\varrho )\nabla
\varrho \right) \cdot \Lambda ^{-s}u\,dx\lesssim \left( \norm{ \nabla\varrho}%
_{L^{2}}^{2}+\norm{\nabla^2\varrho}_{L^{2}}^{2}\right) \norm{\Lambda^{-s}u}%
_{L^{2}}.  \label{1E_s_5}
\end{eqnarray}%
Hence, plugging the estimates \eqref{1E_s_1}--\eqref{1E_s_5} into %
\eqref{1E_s_0}, we deduce \eqref{1E_s}.

Now if $s\in (1/2,3/2)$, we shall estimate the right hand side of %
\eqref{1E_s_0}, $i.e.$, $W_{1}\sim W_{5}$ in a different way. Since $s\in
(1/2,3/2)$, we have that $1/2+s/3<1$ and $2<3/s<6$. Then using the
(different) Sobolev interpolation, we have
\begin{eqnarray}
&&%
\begin{split}
W_{1}& =-\gamma \int_{\mathbb{R}^{3}}\Lambda ^{-s}(\varrho \mathrm{div}%
u)\Lambda ^{-s}\varrho \,dx\lesssim \norm{\varrho}_{L^{3/s}}\norm{\nabla u }%
_{L^{2}}\norm{\Lambda^{-s}\varrho}_{L^{2}} \\
& \lesssim \norm{ \varrho}_{L^{2}}^{s-1/2}\norm{ \nabla\varrho}%
_{L^{2}}^{3/2-s}\norm{\nabla u}_{L^{2}}\norm{\Lambda^{-s}\varrho}_{L^{2}}.
\end{split}
\label{1E_s_12} \\
&&W_{2}=-\gamma \int_{\mathbb{R}^{3}}\Lambda ^{-s}(u\cdot \nabla \varrho
)\Lambda ^{-s}\varrho \,dx\lesssim \norm{ u}_{L^{2}}^{s-1/2}\norm{ \nabla u}%
_{L^{2}}^{3/2-s}\norm{\nabla \varrho}_{L^{2}}\norm{\Lambda^{-s}\varrho}%
_{L^{2}},  \label{1E_s_22} \\
&&W_{3}=-\int_{\mathbb{R}^{3}}\Lambda ^{-s}\left( u\cdot \nabla u\right)
\cdot \Lambda ^{-s}u\,dx\lesssim \norm{ u}_{L^{2}}^{s-1/2}\norm{ \nabla u}%
_{L^{2}}^{3/2-s}\norm{\nabla u}_{L^{2}}\norm{\Lambda^{-s}u}_{L^{2}},
\label{1E_s_32} \\
&&%
\begin{split}
W_{4}& =-\int_{\mathbb{R}^{3}}\Lambda ^{-s}\left( h(\varrho )(\bar{\mu}%
\Delta u+(\bar{\mu}+\bar{\lambda})\nabla \mathrm{div}u)\right) \Lambda
^{-s}u\,dx \\
& \lesssim \norm{ \varrho}_{L^{2}}^{s-1/2}\norm{ \nabla \varrho}%
_{L^{2}}^{3/2-s}\norm{\nabla^2 u }_{L^{2}}\norm{\Lambda^{-s}u}_{L^{2}},
\end{split}
\label{1E_s_42} \\
&&W_{5}=-\int_{\mathbb{R}^{3}}\Lambda ^{-s}\left( f(\varrho )\nabla
\varrho \right) \cdot \Lambda ^{-s}u\,dx\lesssim \norm{ \varrho}%
_{L^{2}}^{s-1/2}\norm{ \nabla \varrho}_{L^{2}}^{3/2-s}\norm{\nabla \varrho}%
_{L^{2}}\norm{\Lambda^{-s}\varrho}_{L^{2}}.  \label{1E_s_52}
\end{eqnarray}%
Hence, plugging the estimates \eqref{1E_s_12}--\eqref{1E_s_52} into %
\eqref{1E_s_0}, we deduce \eqref{1E_s2}.
\end{proof}


\subsection{Proof of Theorem \protect\ref{1mainth}}

\label{sec 3.3} 

In this subsection, we shall combine all the energy estimates that we have
derived in the previous two subsections to prove Theorem \ref{1mainth}.

We first close the energy estimates at each $\ell$-th level in our weaker
sense. Let $N\ge 3$ and $0\le\ell\le m-1$ with $[\frac{N}{2}]+2\le m\le N$.
Summing up the estimates \eqref{1E_k} of Lemma \ref{1Ekle} for from $k=\ell$
to $m-1$, and then adding the resulting estimates with the estimates %
\eqref{1E_k+1} of Lemma \ref{1Ek+1le} for $k=m-1$, by changing the index and
since $\delta$ is small, we obtain
\begin{equation}  \label{1proof1}
\frac{d}{dt}\sum_{\ell\le k\le m}\left(\gamma\norm{\nabla^{k} \varrho}%
_{L^2}^2+\norm{\nabla^{k}u}_{L^2}^2\right) +C_1\sum_{\ell+1\le k\le m+1}%
\norm{\nabla^{k} u}_{L^2}^2 \le C_2 \delta \sum_{\ell+1\le k\le m} %
\norm{\nabla^{k}\varrho}_{L^2}^2.
\end{equation}
Summing up the estimates \eqref{1E_kE_k+1} of Lemma \ref{1EkEk+1le} for from
$k=\ell$ to $m-1$, we have
\begin{equation}  \label{1proof2}
\frac{d}{dt}\sum_{\ell\le k\le m-1}\int_{\mathbb{R}^3}
\nabla^ku\cdot\nabla\nabla^k\varrho\,dx +C_3\sum_{\ell+1\le k\le m}%
\norm{\nabla^{k}\varrho}_{L^2}^2 \le C_4 \sum_{\ell+1\le k\le m+1}%
\norm{\nabla^{k}u}_{L^2}^2.
\end{equation}
Multiplying \eqref{1proof2} by $2C_2\delta/C_3$, adding it with %
\eqref{1proof1}, since $\delta>0$ is small, we deduce that there exists a
constant $C_5>0$ such that for $0\le \ell\le m-1$,
\begin{equation}  \label{1proof3}
\begin{split}
& \frac{d}{dt}\left\{\sum_{\ell\le k\le m} \left(\gamma\norm{\nabla^{k}
\varrho}_{L^2}^2+\norm{\nabla^{k} u}_{L^2}^2\right) +\frac{2C_1\delta}{C_3}%
\sum_{\ell\le k\le m-1}\int_{\mathbb{R}^3}
\nabla^ku\cdot\nabla\nabla^k\varrho\,dx\right\} \\
&\quad+C_5\left\{\sum_{\ell+1\le k\le m} \norm{\nabla^{k}\varrho}%
_{L^2}^2+\sum_{\ell+1\le k\le m+1}\norm{\nabla^{k} u}_{L^2}^2\right\}\le 0.
\end{split}%
\end{equation}

Next, we define $\mathcal{E}_\ell^m(t)$ to be $C_5^{-1}$ times the
expression under the time derivative in \eqref{1proof3}. Observe that since $%
\delta$ is small $\mathcal{E}_\ell^m(t)$ is equivalent to $\norm{\nabla^\ell
\varrho(t)}_{H^{m-\ell}}^2+\norm{\nabla^\ell u(t)}_{H^{m-\ell}}^2$, that is,
there exists a constant $C_6>0$ such that for $0\le \ell\le m-1$,
\begin{equation}  \label{1proof4}
C_6^{-1}\norm{\nabla^\ell \varrho(t)}_{H^{m-\ell}}^2+\norm{\nabla^\ell u(t)}%
_{H^{m-\ell}}^2 \le \mathcal{E}_\ell^m(t)\le C_6\norm{\nabla^\ell \varrho(t)}%
_{H^{m-\ell}}^2+\norm{\nabla^\ell u(t)}_{H^{m-\ell}}^2.
\end{equation}
Then we may write \eqref{1proof3} as that for $0\le \ell\le m-1$,
\begin{equation}  \label{1proof5}
\frac{d}{dt}\mathcal{E}_\ell^m+\norm{\nabla^{\ell+1}\varrho}%
_{H^{m-\ell-1}}^2+\norm{\nabla^{\ell+1}u}_{H^{m-\ell}}^2\le 0.
\end{equation}

\begin{proof}[Proof of \eqref{1HNbound}]
Taking $\ell =0$ and $m=[\frac{N}{2}]+2$ in \eqref{1proof5}, and then
integrating directly in time, in light of \eqref{1proof4}, we obtain
\begin{equation}
\begin{split}
\norm{ \varrho(t)}_{H^{[\frac{N}{2}]+2}}^{2}+\norm{ u(t)}_{H^{[\frac{N}{2}%
]+2}}^{2}& \leq C_{6}\mathcal{E}_{0}^{[\frac{N}{2}]+2}(t)\leq C_{6}\mathcal{E%
}_{0}^{[\frac{N}{2}]+2}(0) \\
& \leq C_{6}^{2}\left( \norm{\varrho_0}_{H^{[\frac{N}{2}]+2}}^{2}+\norm{ u_0}%
_{H^{[\frac{N}{2}]+2}}^{2}\right) .
\end{split}
\label{1proof6}
\end{equation}%
By a standard continuity argument, this closes the a priori estimates %
\eqref{1a priori} if we assume $\norm{ \varrho_0}_{H^{[\frac{N}{2}]+2}}+%
\norm{ u_0}_{H^{[\frac{N}{2}]+2}}\leq \delta _{0}$ is sufficiently small.
This in turn allows us to take $\ell =0$ and $[\frac{N}{2}]+2\leq m\leq N$
in \eqref{1proof5}, and then integrate it directly in time, to obtain
\begin{equation}
\norm{ \varrho(t)}_{H^{m}}^{2}+\norm{ u(t)}_{H^{m}}^{2}+\int_{0}^{t}%
\norm{\nabla\varrho(\tau)}_{H^{m-1}}^{2}+\norm{\nabla u(\tau)}%
_{H^{m}}^{2}\,d\tau \leq C_{4}^{2}\left( \norm{ \varrho_0}_{H^{m}}^{2}+%
\norm{u_0}_{H^{m}}^{2}\right) .  \label{1proof7}
\end{equation}%
This proves \eqref{1HNbound}.
\end{proof}

Now we turn to prove \eqref{1H-sbound}--\eqref{1decay}. However, we are not
able to prove them for all $s\in[0,3/2)$ at this moment. We shall first
prove them for $s\in [0,1/2]$.

\begin{proof}[{Proof of \eqref{1H-sbound}--\eqref{1decay} for {$s\in \lbrack
0,1/2]$}}]
We define $\mathcal{E}_{-s}(t)$ to be the expression under the time
derivative in the estimates \eqref{1E_s}--\eqref{1E_s2} of Lemma \ref{1Esle}%
, which is equivalent to $\norm{ \Lambda^{-s}\varrho(t)}_{L^{2}}^{2}+%
\norm{\Lambda^{-s} u(t)}_{L^{2}}^{2}$. Then, integrating in time \eqref{1E_s}%
, by \eqref{1HNbound}, we obtain that for $s\in (0,1/2]$,
\begin{equation}
\begin{split}
\mathcal{E}_{-s}(t)& \leq {\mathcal{E}_{-s}(0)}+C\int_{0}^{t}\norm{
\nabla\varrho(\tau), \nabla u(\tau)}_{H^{1}}^{2}\sqrt{\mathcal{E}_{-s}(\tau )%
}\,d\tau  \\
& \leq  C_0\left(1+\sup_{0\leq \tau \leq t}\sqrt{\mathcal{E}_{-s}(\tau )}\right).
\end{split}
\label{1-sin2}
\end{equation}%
This implies \eqref{1H-sbound} for $s\in \lbrack 0,1/2]$, that is,
\begin{equation}
\norm{ \Lambda^{-s}\varrho(t)}_{L^{2}}^{2}+\norm{\Lambda^{-s} u(t)}%
_{L^{2}}^{2}\leq C_{0}\ \hbox{ for }s\in \lbrack 0,1/2].
\label{1proof8}
\end{equation}%
If $\ell =1,\dots ,N-1$, we may use Lemma \ref{1-sinte} to have
\begin{equation}
\norm{\nabla^{\ell+1} f}_{L^{2}}\geq C\norm{ \Lambda^{-s} f}_{L^{2}}^{-\frac{%
1}{\ell +s}}\norm{\nabla^\ell f}_{L^{2}}^{1+\frac{1}{\ell +s}}.
\label{1proof10}
\end{equation}%
By this fact and \eqref{1proof8}, we may find
\begin{equation}
\norm{\nabla^{\ell+1}\varrho}_{L^{2}}^{2}+\norm{\nabla^{\ell+1}u}%
_{L^{2}}^{2}\geq C_{0}\left( \norm{\nabla^{\ell}\varrho}_{L^{2}}^{2}+%
\norm{\nabla^{\ell}u}_{L^{2}}^{2}\right) ^{1+\frac{1}{\ell +s}}.
\end{equation}%
This together with \eqref{1HNbound} implies in particular that for $\ell
=1,\dots ,N-1$,
\begin{equation}
\norm{\nabla^{\ell+1}\varrho}_{H^{N-\ell -1}}^{2}+\norm{\nabla^{\ell+1}u}%
_{H^{N-\ell -1}}^{2}\geq C_{0}\left( \norm{\nabla^{\ell}\varrho}_{H^{N-\ell
}}^{2}+\norm{\nabla^{\ell}u}_{H^{N-\ell }}^{2}\right) ^{1+\frac{1}{\ell +s}}.
\label{1inin}
\end{equation}%
In view of \eqref{1proof4} and \eqref{1inin}, we then deduce from %
\eqref{1proof5} with $m=N$ the following time differential inequality
\begin{equation}
\frac{d}{dt}\mathcal{E}_{\ell }^{N}+C_{0}\left( \mathcal{E}_{\ell
}^{N}\right) ^{1+\frac{1}{\ell +s}}\leq 0\ \hbox{ for }\ell =1,\dots ,N-1.
\label{1proof11}
\end{equation}%
Solving this inequality directly gives, together with \eqref{1proof7},
\begin{equation}
\mathcal{E}_{\ell }^{N}(t)\leq C_{0}(1+t)^{-(\ell +s)}\ \hbox{ for }\ell
=1,\dots ,N-1.  \label{1proof12}
\end{equation}

Consequently, in view of \eqref{1proof4}, we obtain from \eqref{1proof12}
that for $s\in \lbrack 0,1/2]$,
\begin{equation}
\norm{\nabla^\ell \varrho(t)}_{H^{N-\ell }}^{2}+\norm{\nabla^\ell u(t)}%
_{H^{N-\ell }}^{2}\leq C_{0}(1+t)^{-(\ell +s)}\ \hbox{ for }\ell =1,\dots
,N-1.  \label{1proof13}
\end{equation}%
Thus, by \eqref{1proof13}, \eqref{1HNbound} and the interpolation, we deduce %
\eqref{1decay} for $s\in \lbrack 0,1/2]$.
\end{proof}

\begin{proof}[Proof of \eqref{1H-sbound}--\eqref{1decay} for {$s\in(1/2,3/2)$%
}]
Notice that the arguments for the case $s\in[0,1/2]$ can not be applied to
this case. However, observing that we have $\varrho_0, u _0\in \dot{H}^{-1/2}
$ since $\dot{H}^{-s}\cap L^2\subset\dot{H}^{-s^{\prime}}$ for any $%
s^{\prime}\in [0,s]$, we then deduce from what we have proved for %
\eqref{1H-sbound} and \eqref{1decay} with $s=1/2$ that the following decay
result holds:
\begin{equation}  \label{1proof14}
\norm{\nabla^\ell \varrho(t)}_{H^{N-\ell}}^2+\norm{\nabla^\ell u(t)}%
_{H^{N-\ell}}^2 \le C_0(1+t)^{-(\ell+1/2)}\ \hbox{ for }-\frac{1}{2}\le
\ell\le N-1.
\end{equation}
Hence, by \eqref{1proof14}, we deduce from \eqref{1E_s2} that for $%
s\in(1/2,3/2)$,
\begin{equation}  \label{1-sin2''}
\begin{split}
\mathcal{E}_{-s}(t) &\le {\mathcal{E}_{-s}(0)}+C\int_0^t \left(\norm{
(\varrho, u)}_{L^2}^{s-1/2}\norm{(\nabla\varrho, \nabla u)}%
_{H^1}^{5/2-s}\right)\sqrt{\mathcal{E}_{-s}(\tau)} \,d\tau \\
&\le C_0+ C_0\int_0^t(1+\tau)^{-(7/4-s/2)}\,d\tau\sup_{0\le\tau\le t}\sqrt{%
\mathcal{E}_{-s}(\tau)} \\
&\le C_0\left(1+\sup_{0\le\tau\le t}\sqrt{\mathcal{E}_{-s}(\tau)}\right).
\end{split}%
\end{equation}
This implies \eqref{1H-sbound} for $s\in(1/2,3/2)$, that is,
\begin{equation}  \label{1sss}
\norm{ \Lambda^{-s}\varrho(t)}_{L^2}^2+\norm{\Lambda^{-s} u(t)}_{L^2}^2 \le
C_0\ \hbox{ for }s\in(1/2,3/2).
\end{equation}

Now that we have proved \eqref{1sss}, we may repeat the arguments leading to %
\eqref{1decay} for $s\in(1/2,3/2)$ to prove that it holds for $s\in(1/2,3/2)$%
. This completes the proof of Theorem \ref{1mainth}.
\end{proof}

\section{Boltzmann equation}

\subsection{Energy estimates}

\label{sec 4.1} 
In this subsection, we will derive the a priori nonlinear energy estimates
for the equation \eqref{perturbation}. We first derive the following standard energy estimates:

\begin{lemma}
\label{lemma linear 1} Let $k=0,\dots,N$, then we have
\begin{equation}  \label{Linear ener 1}
\frac{d}{dt}\norm{\nabla^k f}_{L^2}^2 +\sigma_0\norm{\nabla^k\{{\bf I-P}\} f}%
_\nu^2 \lesssim \sum_{|\gamma_1|\le k}\norm{|\nabla^{|\gamma_1|}
f|_2|\nabla^{k-|\gamma_1|} f|_\nu}_{L^2}^2.
\end{equation}
\end{lemma}

\begin{proof}
Applying $\partial^\gamma$ with $|\gamma|=k$, $0\le k\le N$ to %
\eqref{perturbation}, multiplying the resulting identity by $\partial^\gamma
f$ and then integrating over $\mathbb{R}^3_x\times \mathbb{R}^3_v$, we
obtain
\begin{equation}  \label{EL11}
\frac{1}{2}\frac{d}{dt}\norm{\partial^\gamma f}_{L^2}^2 + (L\partial^\gamma f
,\partial^\gamma f ) =(\partial^\gamma\Gamma(f,f), \partial^\gamma f) .
\end{equation}
The estimate \eqref{lineares1} of Lemma \ref{linearcol1} implies
\begin{equation}  \label{EL12}
(L\partial^\gamma f ,\partial^\gamma f )\ge \sigma_0\norm{\partial^\gamma\{{%
\bf I-P}\} f}_\nu^2.
\end{equation}
On the other hand, by the collision invariant property, the estimate %
\eqref{nonlineares3} (with $\eta=1/2$) of Lemma \ref{nonlinearcol1} and
symmetry, together with Cauchy's inequality, we obtain
\begin{equation}  \label{EL13}
\begin{split}
(\partial^\gamma\Gamma(f,f),\partial^\gamma f)& =\sum_{\gamma_1\le
\gamma}C_\gamma^{\gamma_1}(\Gamma(\partial^{\gamma-\gamma_1} f,
\partial^{\gamma_1} f),\partial^\gamma \{\mathbf{I-P}\}f) \\
&\le C\sum_{\gamma_1\le \gamma}|\nu^{-1/2}\Gamma(\partial^{\gamma-\gamma_1}
f,\partial^{\gamma_1} f)|_2|\nu^{1/2}\partial^\gamma \{\mathbf{I-P}\}f|_2 \\
&\le C\sum_{\gamma_1\le \gamma}\int_{\mathbb{R}^3_x}|\partial^{\gamma_1}
f|_2|\partial^{\gamma-\gamma_1} f|_\nu|\partial^\gamma \{\mathbf{I-P}%
\}f|_\nu\,dx \\
&\le C\sum_{|\gamma_1|\le k}\norm{|\nabla^{|\gamma_1|}
f|_2|\nabla^{k-|\gamma_1|} f|_\nu}_{L^2}^2 +\frac{\sigma_0}{2}\norm{\partial^%
\gamma \{{\bf I-P}\}f}_\nu^2.
\end{split}%
\end{equation}
Hence, by the estimates \eqref{EL12}--\eqref{EL13}, we deduce \eqref{Linear
ener 1} from \eqref{EL11}.
\end{proof}

Next, notice that the dissipation estimate in \eqref{Linear ener 1} is
degenerate, and it only controls the microscopic part $\{\mathbf{I- P}\}f$.
Hence, in order to get the full dissipation estimate we shall use the
equation \eqref{perturbation} to estimate the hydrodynamic part $\mathbf{P}f$
in terms of the microscopic part.

\begin{lemma}
\label{lemma linear 2} Let $N\ge 3$. If $\norm{f(t)}_{L^2_vH^{N}_x}\lesssim
\delta$, then we have that for $k=0,\dots, N-1$, there exists a constant $%
C_1>0$ such that
\begin{equation}  \label{linear ener 2}
\frac{dG_k}{dt} +\norm{\nabla^{k+1} \mathbf{P }f }_{L^2}^2 \le C_1\left\{\norm{\nabla^{k}\{%
\mathbf{I-P }\} f}_{L^2}^2+\norm{\nabla^{k+1}\{\mathbf{I-P }\} f }_{L^2}^2\right\}.
\end{equation}
Here $G_k(t)$ is defined by \eqref{G_k(t)} with the property
\begin{equation}  \label{G_k estimate}
|G_k|\lesssim\norm{\nabla^k f}_{L^2}^2+\norm{\nabla^{k+1} f}_{L^2}^2.
\end{equation}
\end{lemma}

\begin{proof}
We represent $\mathbf{P}f$ as
\begin{equation}  \label{macro}
\mathbf{P}f =\left\{a (t,x)+b (t,x)\cdot v+c(t,x)|v|^2\right\}\sqrt{\mu}.
\end{equation}
Then for each $k=0,\dots,N-1$, it follows from Lemma 6.1 of \cite{G2006}
that there exists a constant $C>0$ such that
\begin{equation}  \label{macro micro}
\norm{\nabla^{k+1} \mathbf{P }f }_{L^2}^2 \le\frac{dG_k}{dt} +C\left\{\norm{\nabla^{k}\{%
\mathbf{I-P }\} f }_{L^2}^2+\norm{\nabla^{k+1}\{\mathbf{I-P }\} f}_{L^2}^2\right\} +C%
\norm{\left\langle\nabla^k\Gamma(f,f),\zeta\right\rangle}_{L^2}^2,
\end{equation}
where $G_k(t)$ is defined as
\begin{equation}  \label{G_k(t)}
\begin{split}
G_k(t)&=\sum_{|\gamma|=k}\int_{\mathbb{R}^3}\left\{\langle\{\mathbf{I-P }%
\}\partial^\gamma f,\zeta_{a}\rangle\cdot \nabla_x\partial^\gamma a +\{%
\mathbf{I-P}\}\partial^\gamma f
,\zeta_{ij}\rangle\cdot\partial_j\partial^\gamma b_i \right\} dx \\
&\quad+\sum_{|\gamma|=k}\int_{\mathbb{R}^3}\left\{\langle\{\mathbf{I-P }%
\}\partial^\gamma f ,\zeta_c\rangle\cdot\nabla_x\partial^\gamma c +
\partial^\gamma b\cdot\nabla_x\partial^\gamma a \right\} \,dx,
\end{split}%
\end{equation}
and $\zeta$, $\zeta_a(v)$, $\zeta_{i j} (v)$, and $\zeta_c(v)$ are some
fixed linear combinations of the basis
\begin{equation*}
[\sqrt{\mu},v_i\sqrt{\mu},v_iv_j\sqrt{\mu},v_i|v|^2\sqrt{\mu}],\ 1\le i,j\le
3.
\end{equation*}
The proof of \eqref{macro micro} is based on the use of the \textit{local
conservation laws} and \textit{macroscopic equations} which are derived from
the so called macro-micro decomposition.

We now estimate the nonlinear term in the right hand side of \eqref{macro
micro}. By the estimate \eqref{nonlineares2} of Lemma \ref{nonlinearcol1}
and the fact that $\zeta$ decays exponentially in $v$, we have
\begin{equation}  \label{positive f}
\norm{\left\langle\nabla^k\Gamma(f,f),\zeta\right\rangle}_{L^2}^2\lesssim\sum_{|%
\gamma|=k}\sum_{\gamma_1\le \gamma}\norm{ \langle\Gamma(\partial^{\gamma_1}
f,\partial^{\gamma-\gamma_1}f),\zeta\rangle }_{L^2}^2\lesssim
\sum_{|\gamma_1|\le k}\norm{|\nabla^{|\gamma_1|}
f|_2|\nabla^{k-|\gamma_1|}f|_2}_{L^2}^2.
\end{equation}
By H\"older's inequality, Minkowski's integral inequality \eqref{min es} of
Lemma \ref{Minkowski}, the Sobolev interpolation of Lemma \ref{interpolation}
and Young's inequality, we obtain
\begin{equation}  \label{N2 0}
\begin{split}
\norm{|\nabla^{|\gamma_1|} f|_2|\nabla^{k-|\gamma_1|}f|_2}_{L^2} &\lesssim %
\norm{\nabla^{|\gamma_1|} f}_{L^6_xL^2_v}\norm{ \nabla^{k-|\gamma_1|}f}%
_{L^3_xL^2_v}\lesssim \norm{\nabla^{|\gamma_1|} f}_{L^2_vL^6_x}\norm{
\nabla^{k-|\gamma_1|}f}_{L^2_vL^3_x} \\
&\lesssim \norm{f}_{L^2}^{1-\frac{|\gamma_1|+1}{k+1}}\norm{\nabla^{k+1}f}_{L^2}^{\frac{%
|\gamma_1|+1}{k+1}} \norm{\nabla^\alpha f}_{L^2}^{\frac{|\gamma_1|+1}{k+1}}%
\norm{\nabla^{k+1}f}_{L^2}^{1-\frac{|\gamma_1|+1}{k+1}} \\
&\lesssim \delta \norm{\nabla^{k+1}f}_{L^2}.
\end{split}%
\end{equation}
Here we have denoted $\alpha$ by
\begin{equation}
\begin{split}
&\frac{1}{3}-\frac{k-|\gamma_1|}{3}=\left(\frac{1}{2}-\frac{\alpha}{3}%
\right)\times \frac{|\gamma_1|+1}{k+1}+\left(\frac{1}{2}-\frac{k+1}{3}%
\right)\times\left(1-\frac{|\gamma_1|+1}{k+1}\right) \\
&\quad\Longrightarrow\alpha=\frac{k+1}{2(|\gamma_1|+1)}\le \frac{k+1}{2}\le
\frac{N}{2}.
\end{split}%
\end{equation}
Hence, we have
\begin{equation}  \label{4.16}
\norm{\left\langle\nabla^k\Gamma(f,f),\zeta\right\rangle}_{L^2}^2\lesssim\delta^2 %
\norm{\nabla^{k+1}f}_{L^2}^2.
\end{equation}
Plugging the estimates \eqref{4.16} into \eqref{macro micro}, since $\delta$
is small, we obtain \eqref{linear ener 2}.
\end{proof}

To conclude our energy estimates, we turn to the nonlinear term in the right
hand side of \eqref{Linear ener 1}. We shall derive the following two sets of nonlinear
estimates, depending on whether we assume the weighted norm of initial data.

\begin{lemma}
\label{nonlinear lemma 2} Let $N\ge 3$. If $\norm{f(t)}_{L^2_vH^{N}_x}%
\lesssim \delta$, then for $k=0,\dots,N-1$, we have
\begin{equation}  \label{nonlinear estimate 21}
\sum_{|\gamma_1|\le k}\norm{|\nabla^{|\gamma_1|} f|_2|\nabla^{k-|\gamma_1|}
f|_\nu}_{L^2}^2 \lesssim\delta^2 \left(\norm{\nabla^{k+1}f}_{L^2}^2+\sum_{1\le \ell \le
N}\norm{\nabla^{\ell} \{{\bf I- P}\}f}_\nu^2\right);
\end{equation}
and for $k=N$, we have
\begin{equation}  \label{nonlinear estimate 22}
\sum_{|\gamma_1|\le N}\norm{|\nabla^{|\gamma_1|} f|_2|\nabla^{N-|\gamma_1|}
f|_\nu}_{L^2}^2 \lesssim\delta^2 \left(\norm{\nabla^N f}_{L^2}^2+\sum_{1\le \ell \le N}%
\norm{\nabla^{\ell} \{{\bf I- P}\}f}_\nu^2\right).
\end{equation}
\end{lemma}

\begin{proof}
We first use the splitting $f=\mathbf{P}f+\{\mathbf{I-P}\}f$ to have
\begin{equation}  \label{444}
\begin{split}
\norm{|\nabla^{|\gamma_1|} f|_2|\nabla^{k-|\gamma_1|} f|_\nu }_{L^2}&\lesssim%
\norm{|\nabla^{|\gamma_1|} f|_2|\nabla^{k-|\gamma_1|} f|_2 }_{L^2} +%
\norm{|\nabla^{|\gamma_1|} f|_2|\nabla^{k-|\gamma_1|}\{{\bf I-P }\} f|_\nu}_{L^2}
\\
&:= J_{11}+J_{12}.
\end{split}%
\end{equation}
For the term $J_{11}$, if $k=0,\dots,N-1$, it has been already bounded in %
\eqref{N2 0} as
\begin{equation}  \label{N2 01}
J_{11}\lesssim \delta \norm{\nabla^{k+1}f}_{L^2};
\end{equation}
while for $k=N$, by the symmetry, we may assume $|\gamma_1|\le \frac{N}{2}$
to obtain, by Lemma \ref{Minkowski} and Lemma \ref{interpolation},
\begin{equation}  \label{temple}
\begin{split}
J_{11}&\lesssim\norm{\nabla^{|\gamma_1|}f}_{L^\infty_xL^2_v}%
\norm{\nabla^{N-|\gamma_1|}f}_{L^2}\lesssim\norm{\nabla^{|\gamma_1|}f}%
_{L^2_vL^\infty_x}\norm{\nabla^{N-|\gamma_1|}f}_{L^2} \\
&\lesssim \norm{\nabla^{\alpha}f}_{L^2}^{1-\frac{|\gamma_1|}{N}}\norm{\nabla^{N}f}_{L^2}%
^{\frac{|\gamma_1|}{N}}\norm{f}_{L^2}^{\frac{|\gamma_1|}{N}}\norm{ \nabla^Nf}_{L^2}^{1-%
\frac{|\gamma_1|}{N}} \\
&\lesssim \delta\norm{\nabla^{N}f}_{L^2},
\end{split}%
\end{equation}
where we have denoted $\alpha$ by
\begin{equation}
\begin{split}
&-\frac{|\gamma_1|}{3}=\left(\frac{1}{2}-\frac{\alpha}{3}\right)\times%
\left(1-\frac{|\gamma_1|}{N}\right)+\left(\frac{1}{2}-\frac{N}{3}%
\right)\times\frac{|\gamma_1|}{N} \\
&\quad\Longrightarrow\alpha=\frac{3N}{2(N-|\gamma_1|)}\le 3\ \text{ since }%
|\gamma_1|\le N/2.
\end{split}%
\end{equation}

Now for the term $J_{12}$, note that we can only bound the $\nu$-weighted
factor by the dissipation, so we can not pursue as before to adjust the
index. Notice that $\{\mathbf{I-P }\}f$ is always part of the dissipation.
If $|\gamma_1|\le k-2$ (if $k-2<0$, then it's nothing in this case, etc.)
and hence $k-|\gamma_1|\ge 2$, then we bound
\begin{equation}
J_{12}\le \norm{\nabla^{|\gamma_1|}f}_{L^\infty_xL^2_v}\norm{\nabla^{k-|%
\gamma_1|}\{{\bf I-P }\}f}_\nu \lesssim \delta \sum_{2\le \ell\le N}%
\norm{\nabla^\ell \{{\bf I-P }\}f}_\nu;
\end{equation}
and if $|\gamma_1|\ge k-1$ and hence $k-|\gamma_1|\le 1$, then we bound, by
Sobolev's inequality,
\begin{equation}
\begin{split}
J_{12}&\le \norm{\nabla^{|\gamma_1|}f}_{L^2}\norm{|\nabla^{k-|\gamma_1|}\{{\bf I-P
}\}f|_\nu}_{L^\infty_x} \le \norm{\nabla^{|\gamma_1|}f}_{L^2}\norm{\nu^{1/2}%
\nabla^{k-|\gamma_1|}\{{\bf I-P }\}f}_{L^2_vL^\infty_x} \\
&\lesssim \delta \sum_{1\le \ell\le 3}\norm{\nabla^\ell \{{\bf I-P }\}f}_\nu.
\end{split}%
\end{equation}
Hence, we have that for $k=0,\dots,N$,
\begin{equation}  \label{j12}
J_{12}\lesssim \delta\sum_{1\le \ell\le N}\norm{\nabla^\ell \{{\bf I-P }\}f}%
_\nu.
\end{equation}

Consequently, in light of the estimates \eqref{N2 01}, \eqref{temple} and %
\eqref{j12}, we then get \eqref{nonlinear estimate 21}--\eqref{nonlinear
estimate 22}.
\end{proof}

\begin{lemma}
\label{nonlinear lemma 3} Let $N\ge 3$. If $\norm{f(t)}_{L^2_vH^{N}_x}+%
\norm{f(t)}_\nu\lesssim \delta$, then for $k=0,\dots,N-1$, we have
\begin{equation}  \label{nonlinear estimate 31}
\sum_{|\gamma_1|\le k}\norm{|\nabla^{|\gamma_1|} f|_2|\nabla^{k-|\gamma_1|}
f|_\nu}_{L^2}^2 \lesssim\delta^2 \left(\norm{\nabla^{k+1}f}_{L^2}^2+ \norm{\nabla^k
\{{\bf I-P}\}f}_\nu^2 \right);
\end{equation}
and for $k=N$, we have
\begin{equation}  \label{nonlinear estimate 32}
\sum_{|\gamma_1|\le N}\norm{|\nabla^{|\gamma_1|} f|_2|\nabla^{N-|\gamma_1|}
f|_\nu}_{L^2}^2 \lesssim\delta^2 \norm{\nabla^N f}_\nu^2.
\end{equation}
\end{lemma}

\begin{proof}
Clearly, we only need to revise the estimates of the term $J_{12}$ defined
in \eqref{444}. Note that now we can also bound the $\nu$-weighted factor by
the energy, so we can pursue to adjust the index. For $k=0,\dots,N-1$, if $%
|\gamma_1|=0$, then we have
\begin{equation}
J_{12}\lesssim \norm{f}_{L^\infty_xL^2_v}\norm{\nabla^k\{{\bf I-P }\}f}%
_{\nu} \lesssim \delta \norm{\nabla^{k}\{{\bf I-P }\}f}_\nu;
\end{equation}
if $|\gamma_1|\ge 1$, then by Lemma \ref{Minkowski} and Lemma \ref%
{interpolation}, we have
\begin{equation}
\begin{split}
J_{12}&\lesssim \norm{\nabla^{|\gamma_1|} f}_{L^3_xL^2_v}\norm{\nu^{1/2}
\nabla^{k-|\gamma_1|}\{{\bf I-P }\}f}_{L^6_xL^2_v} \\
&\lesssim \norm{\nabla^{|\gamma_1|} f}_{L^2_vL^3_x}\norm{\nu^{1/2}
\nabla^{k-|\gamma_1|}\{{\bf I-P }\}f}_{L^2_vL^6_x} \\
&\lesssim \norm{\nabla^\alpha f}_{L^2}^{1-\frac{|\gamma_1|-1}{k}}%
\norm{\nabla^{k+1}f}_{L^2}^{\frac{|\gamma_1|-1}{k}} \norm{\{{\bf I-P }\}f}_\nu^{%
\frac{|\gamma_1|-1}{k}}\norm{\nabla^{k}\{{\bf I-P }\}f}_\nu^{1-\frac{%
|\gamma_1|-1}{k}} \\
&\lesssim \delta\left(\norm{\nabla^{k+1}f}_{L^2}+\norm{\nabla^{k}\{{\bf I-P }\}f}%
_\nu\right).
\end{split}%
\end{equation}
where we have denoted $\alpha$ by
\begin{equation}
\begin{split}
&\frac{1}{3}-\frac{|\gamma_1|}{3}= \left(\frac{1}{2}-\frac{\alpha}{3}%
\right)\times \left(1-\frac{|\gamma_1|-1}{k}\right)+\left(\frac{1}{2}-\frac{%
k+1}{3}\right)\times \frac{|\gamma_1|-1}{k} \\
&\quad\Longrightarrow \alpha=\frac{\frac{3}{2}k-(|\gamma_1|-1)}{%
k-(|\gamma_1|-1)}\le \frac{k}{2}+1\le \frac{N+1}{2}.
\end{split}%
\end{equation}
Hence, we have that for $k=0,\dots,N-1$,
\begin{equation}  \label{j1211}
J_{12}\lesssim \delta\left(\norm{\nabla^{k+1}f}_{L^2}+\norm{\nabla^{k}\{{\bf I-P
}\}f}_\nu\right).
\end{equation}
This together with \eqref{N2 01} implies \eqref{nonlinear estimate 31}.

Now for $k=N$, if $|\gamma_1|\ge N-1$, by Lemma \ref{Minkowski} and Lemma %
\ref{interpolation}, we estimate
\begin{equation}
\begin{split}
J_{12}&\le \norm{\nabla^{|\gamma_1|} f}_{L^2}\norm{\nu^{1/2}\nabla^{N-|\gamma_1|}%
\{{\bf I-P }\} f}_{L^\infty_xL^2_v} \le\norm{\nabla^{|\gamma_1|} f}_{L^2}%
\norm{\nu^{1/2}\nabla^{N-|\gamma_1|}\{{\bf I-P }\} f}_{L^2_vL^\infty_x} \\
&\lesssim \norm{\nabla^\alpha f}_{L^2}^{1-\frac{2|\gamma_1|-3}{2N}}\norm{\nabla^Nf}_{L^2}%
^{\frac{2|\gamma_1|-3}{2N}}\norm{\{{\bf I-P }\}f}_{\nu}^{\frac{2|\gamma_1|-3%
}{2N}}\norm{\nabla^N\{{\bf I-P }\}f}_{\nu}^{1-\frac{2|\gamma_1|-3}{2N}} \\
&\lesssim \delta\left(\norm{\nabla^Nf}_{L^2}+\norm{\nabla^N\{{\bf I-P }\}f}%
_{\nu}\right),
\end{split}%
\end{equation}
where we have denoted $\alpha$ by
\begin{equation}
\begin{split}
&|\gamma_1|=\alpha\times \left(1-\frac{2|\gamma_1|-3}{2N}\right)+N\times%
\frac{2|\gamma_1|-3}{2N} \\
&\quad\Longrightarrow \alpha=\frac{3N}{2(N-|\gamma_1|)+3}\le N;
\end{split}%
\end{equation}
and if $|\gamma_1|\le N-2$, by again Lemma \ref{Minkowski} and Lemma \ref%
{interpolation}, we estimate
\begin{equation}
\begin{split}
J_{12}&\le \norm{\nabla^{|\gamma_1|} f}_{L^\infty_xL^2_v}\norm{\nabla^{N-|%
\gamma_1|}\{{\bf I-P }\} f}_{\nu} \le\norm{\nabla^{|\gamma_1|} f}%
_{L^2_vL^\infty_x}\norm{\nabla^{N-|\gamma_1|}\{{\bf I-P }\} f}_{\nu} \\
&\lesssim \norm{\nabla^\alpha f}_{L^2}^{1-\frac{|\gamma_1|}{N}}\norm{\nabla^{N} f}_{L^2}%
^{\frac{|\gamma_1|}{N}} \norm{\{{\bf I-P }\} f}_\nu^{\frac{|\gamma_1|}{N}}%
\norm{\nabla^{N}\{{\bf I-P }\} f}_\nu^{\frac{N-|\gamma_1|}{N}} \\
&\lesssim \delta\left(\norm{\nabla^N f}_{L^2}+\norm{\nabla^N\{{\bf I-P }\}f}%
_{\nu}\right),
\end{split}%
\end{equation}
where we have denoted $\alpha$ by
\begin{equation}
\begin{split}
&-\frac{|\gamma_1|}{3}= \left(\frac{1}{2}-\frac{N}{3}\right)\times\frac{%
|\gamma_1|}{N}+\left(\frac{1}{2}-\frac{\alpha}{3}\right)\times\left(1-\frac{%
|\gamma_1|}{N}\right) \\
&\quad\Longrightarrow \alpha=\frac{3N}{2(N-|\gamma_1|)}\le \frac{3N}{4}\
\text{ since }|\gamma_1|\le N-2.
\end{split}%
\end{equation}
Hence, we have that for $k=N$,
\begin{equation}  \label{j1212}
J_{12}\lesssim \delta\left(\norm{\nabla^N f}_{L^2}+\norm{\nabla^N\{{\bf I-P }\}f}%
_\nu\right).
\end{equation}
This together with \eqref{temple} implies \eqref{nonlinear estimate 32}.
\end{proof}


\subsection{Energy evolution of of negative Sobolev norms}

\label{sec 4.2} 
In this subsection, we will derive the evolution of the negative Sobolev
norms of the solution. In order to estimate the nonlinear terms, we need to
restrict ourselves to that $s\in (0,3/2)$. We will establish the following
lemma.

\begin{lemma}
\label{lemma H-s} For $s\in (0, 1/2]$, we have
\begin{equation}  \label{H-s1}
\frac{d}{dt}\norm{\Lambda^{-s}f}_{L^2}^2+\sigma_0\norm{\Lambda^{-s}\{{\bf I-P}\}f}%
_\nu^2 \lesssim \norm{\{{\bf I-P}\}f }_\nu^2+ \norm{ \nabla f }_{L^2}^2+ \norm{
\nabla^2 f }_{L^2}^2;
\end{equation}
and for $s\in (1/2, 3/2)$, we have
\begin{equation}  \label{H-s2}
\frac{d}{dt}\norm{\Lambda^{-s}f}_{L^2}^2+\sigma_0\norm{\Lambda^{-s}\{{\bf I-P}\}f}%
_\nu^2\lesssim \norm{\{{\bf I-P}\}f }_\nu^2+\norm{ f }_{L^2}^{2s+1}\norm{ \nabla f
}_{L^2}^{3-2s}.
\end{equation}
\end{lemma}

\begin{proof}
Applying $\Lambda^{-s}$ to \eqref{perturbation}, and then taking the $L^2$
inner product with $\Lambda^{-s}f$, together with the collision invariant
property and Cauchy's inequality, we have
\begin{equation}  \label{Hs3}
\begin{split}
\frac{1}{2}\frac{d}{dt}\norm{\Lambda^{-s}f}_{L^2}^2 + \sigma_0\norm{\Lambda^{-s}%
\{{\bf I-P}\} f}_\nu^2 &\le (\Lambda^{-s}\Gamma(f,f), \Lambda^{-s}\{\mathbf{%
I-P}\}f) \\
& \le C\norm{\Lambda^{-s}\left(\nu^{-\frac{1}{2}}\Gamma(f,f)\right)}_{L^2}^2+\frac{%
\sigma_0}{2}\norm{\Lambda^{-s}\{{\bf I-P}\}f}_\nu^2.
\end{split}%
\end{equation}

To estimate the right hand side of \eqref{Hs3}, since $0<s<3/2$, we let $%
1<p<2$ to be with ${1}/{2}+{s}/{3}={1}/{p}$. By the estimate \eqref{1Riesz
es} of Riesz potential in Lemma \ref{1Riesz}, Minkowski's integral
inequality \eqref{min es} of Lemma \ref{Minkowski}, and the estimate %
\eqref{nonlineares3} (with $\eta=1/2$) of Lemma \ref{nonlinearcol1},
together with H\"older's inequality, we obtain
\begin{equation}  \label{Hs es 5}
\begin{split}
&\norm{\Lambda^{-s}\left(\nu^{-\frac{1}{2}}\Gamma(f,f)\right)}_{L^2} = %
\norm{\Lambda^{-s}\left(\nu^{-\frac{1}{2}}\Gamma(f,f)\right)}_{L_v^2L_x^2}^2
\\
&\quad\lesssim \norm{\nu^{-\frac{1}{2}}\Gamma(f,f)}_{L_v^2L_x^{p}} \le %
\norm{\nu^{-\frac{1}{2}}\Gamma(f,f)}_{L_x^{p}L_v^2} \\
&\quad\lesssim \norm{|f|_2|f |_\nu}_{L_x^{p}} \le \norm{f}_{L_x^\frac{3}{s}%
L_v^2}\norm{f }_\nu \\
&\quad\le \norm{f}_{L_v^2L_x^\frac{3}{s}}\left(\norm{\{{\bf I- P}\}f }_\nu+%
\norm{f}_{L^2}\right).
\end{split}%
\end{equation}
We bound the first term in \eqref{Hs es 5} as, since $3/s>2$, by Sobolev's
inequality,
\begin{equation}  \label{Hs es 6}
\begin{split}
\norm{f}_{L_v^2L_x^\frac{3}{s}}\norm{\{{\bf I- P}\}f }_\nu&\lesssim \norm{f}%
_{L_v^2H_x^2} \norm{\{{\bf I-P}\}f }_\nu \lesssim \delta \norm{\{{\bf
I-P}\}f }_\nu.
\end{split}%
\end{equation}
While for the other term in \eqref{Hs es 5}, we shall separate the estimates
according the value of $s$. If $0< s\le 1/2$, then $3/s\ge 6$, we use the
Sobolev interpolation and Young's inequality to have
\begin{equation}  \label{Hs es 7}
\norm{f}_{L_v^2L_x^\frac{3}{s}}\norm{f}_{L^2} \le \norm{ \nabla f }_{L^2}^{1+s/2}\norm{
\nabla^2 f }_{L^2}^{1-s/2}\norm{ f }_{L^2} \lesssim \delta \left( \norm{ \nabla f }_{L^2}+ %
\norm{\nabla^2 f}_{L^2}\right);
\end{equation}
and if $s\in(1/2,3/2)$, then $2<3/s<6$, we use the (different) Sobolev
interpolation and H\"older's inequality to have
\begin{equation}  \label{Hs es 8}
\norm{f}_{L_v^2L_x^\frac{3}{s}}\norm{f}_{L^2}\lesssim \norm{ f }_{L^2}^{s-1/2}%
\norm{\nabla f }_{L^2}^{3/2-s} \norm{ f }_{L^2}=\norm{ f }_{L^2}^{s+1/2}\norm{\nabla f }_{L^2}%
^{3/2-s}.
\end{equation}

Consequently, in light of \eqref{Hs es 5}--\eqref{Hs es 8}, we deduce from %
\eqref{Hs3} that \eqref{H-s1} holds for $s\in(0,1/2]$ and that \eqref{H-s2}
holds for $s\in(1/2,3/2)$.
\end{proof}


\subsection{Energy evolution of the microscopic part $\{\mathbf{I-P}\}f$}

\label{sec 4.3} 
In this subsection, we will derive the energy evolution of the weighted norm
of the microscopic part. With the help of this weighted norm, we can prove a
further estimates of the microscopic part which allows us to prove the
faster decay of it. The following lemma provides the energy evolution for $\{%
\mathbf{I-P}\}f$.

\begin{lemma}
\label{lemma micro 2} If $\norm{f(t)}_{L^2_vH^{N}_x}\lesssim \delta$, then
we have
\begin{equation}  \label{micro estimate 1}
\frac{d}{dt}\norm{\{\mathbf{I-P }\}f }_{L^2}^2+\sigma_0\norm{\{\mathbf{I-P }\}f }_\nu
^2\lesssim \norm{\nabla f}_{L^2}^2.
\end{equation}
and
\begin{equation}  \label{micro estimate 2}
\frac{d}{dt}\norm{\{\mathbf{I-P }\}f }_\nu^2+\frac{1}{2}\norm{\nu\{\mathbf{I-P }\}f
}_{L^2}^2\lesssim \norm{\nabla f}_{L^2}^2+\norm{\{\mathbf{I-P }\}f }_\nu^2.
\end{equation}
\end{lemma}

\begin{proof}
We only prove \eqref{micro estimate 2}, but the proof of \eqref{micro
estimate 1} is similar. Applying the projection $\{\mathbf{I-P }\}$ to %
\eqref{perturbation}, we obtain
\begin{equation}  \label{micro equation}
\partial_t\{\mathbf{I-P }\}f + v\cdot\nabla_x\{\mathbf{I-P }\} f + L \{%
\mathbf{I-P }\}f = \Gamma(f ,f ) - v\cdot\nabla_x\mathbf{P }f + \mathbf{P }%
(v\cdot\nabla_xf ).
\end{equation}
Taking the $L^2$ inner product of \eqref{micro equation} with $\nu \{\mathbf{%
I-P }\}f $, we have
\begin{equation}  \label{ME20}
\begin{split}
&\frac{1}{2}\frac{d}{dt}\|\{\mathbf{I-P }\}f \|_\nu^2 + (\nu L \{\mathbf{I-P
}\}f ,\{\mathbf{I-P }\}f) \\
&\quad= (\Gamma(f ,f ),\nu\{\mathbf{I-P }\}f ) + (-v\cdot\nabla_x\mathbf{P }%
f +\mathbf{P }(v\cdot\nabla_xf ),\nu\{\mathbf{I-P }\}f ).
\end{split}%
\end{equation}
The estimate \eqref{lineares2} of Lemma \ref{linearcol1} implies
\begin{equation}  \label{ME21}
(\nu L \{\mathbf{I-P }\}f ,\{\mathbf{I-P }\}f )\ge \frac{1}{2}\norm{\nu\{\mathbf{%
I-P }\}f }_{L^2}^2-C\|\{\mathbf{I-P }\}f \|_\nu^2.
\end{equation}
While we use H\"older's inequality, the estimate \eqref{nonlineares3} (with $%
\eta=0$) of Lemma \ref{nonlinearcol1} and Sobolev's inequality to bound
\begin{equation}  \label{ME22}
\begin{split}
(\Gamma(f ,f ),\nu\{\mathbf{I-P }\}f ) &\le C\int_{\mathbb{R}_x^3}|\nu f
|_2|f |_2|\nu\{\mathbf{I-P }\}f |_2\,dx \\
&\le C\int_{\mathbb{R}_x^3}\left\{\nu \{\mathbf{I-P }\}f |_2+|\mathbf{P } f
|_2\right\}|f |_2|\nu\{\mathbf{I-P }\}f |_2\,dx \\
&\le C \norm{f }_{L^\infty_xL^2_v}\norm{\nu \{\mathbf{I-P }\}f}_{L^2}^2+\norm{f}%
_{L^3_xL^2_v}\norm{f}_{L^6_xL^2_v}\norm{\nu\{\mathbf{I-P }\}f}_{L^2} \\
&\le C\delta\left(\norm{\nu\{\mathbf{I-P }\}f }_{L^2}^2+ \norm{\nabla f}_{L^2}^2\right),
\end{split}%
\end{equation}
On the other hand, by the direct computation we can bound the last two terms
in \eqref{ME20} by
\begin{equation}  \label{ME23}
(-v\cdot\nabla_x\mathbf{P }f + \mathbf{P }(v\cdot\nabla_xf ),\nu\{\mathbf{%
I-P }\}f)\le C\norm{\nabla f}_{L^2}^2+\frac{1}{8}\norm{\nu\{\mathbf{I-P }\}f }_{L^2}^2.
\end{equation}
Hence, plugging \eqref{ME21}--\eqref{ME23} into \eqref{ME20}, since $\delta$
is small, we obtain \eqref{micro estimate 2}.
\end{proof}

By \eqref{micro estimate 2}, we know that if $\norm{f_0}_\nu$ is small, then
$\norm{f(t)}_\nu$ is small. With the help of this weighted bound, we can
prove the following energy evolution for $\nabla^k\{\mathbf{I-P }\}f,\
k=1,\dots,N-2$.

\begin{lemma}
\label{lemma micro 3} If $\norm{f(t)}_{L^2_vH^{N}_x}+\norm{f(t)}%
_{L^2_\nu}\lesssim \delta$, then for $k=0,\dots,N-2$, we have
\begin{equation}  \label{micro estimate 3}
\frac{d}{dt}\norm{\nabla^k\{\mathbf{I-P }\}f }_{L^2}^2+\sigma_0\norm{\nabla^k\{\mathbf{%
I-P }\}f }_\nu^2 \lesssim \norm{\nabla^{k+1} f}_{L^2}^2.
\end{equation}
\end{lemma}

\begin{proof}
Applying $\nabla^k$ with $k=0,\dots,N-2$ to \eqref{micro equation} and then
taking the $L^2$ inner product with $\nabla^k\{\mathbf{I-P }\}f $, we have
\begin{equation}  \label{ME30}
\begin{split}
&\frac{1}{2}\frac{d}{dt}\norm{\nabla^k\{\mathbf{I-P }\}f}_{L^2}^2 +
\sigma_0\norm{\nabla^k\{\mathbf{I-P }\}f }_\nu^2 \\
&\quad\le \left(\nabla^k\Gamma(f ,f ),\nabla^k\{\mathbf{I-P }\}f \right) +
\left(-v\cdot\nabla_x\nabla^k\mathbf{P }f +\mathbf{P }(v\cdot\nabla_x%
\nabla^kf ),\nabla^k\{\mathbf{I-P }\}f \right).
\end{split}%
\end{equation}
By the estimate \eqref{nonlinear estimate 31} of Lemma \ref{nonlinear lemma
3}, we may bound
\begin{equation}  \label{ME32}
\left(\nabla^k\Gamma(f ,f ),\nabla^k\{\mathbf{I-P }\}f \right) \le C\delta^2 %
\norm{\nabla^{k+1}f}_{L^2}^2+\left(C\delta^2+\frac{\sigma_0}{8}\right)%
\norm{\nabla^k \{{\bf I-P}\}f}_\nu^2.
\end{equation}
By the direct computation we can bound the last two terms in \eqref{ME30} by
\begin{equation}  \label{ME33}
\left(-v\cdot\nabla_x\nabla^k\mathbf{P }f + \mathbf{P }(v\cdot\nabla_x%
\nabla^k f ),\nabla^k\{\mathbf{I-P }\}f\right)\le C\norm{\nabla^{k+1} f}_{L^2}^2+%
\frac{\sigma_0}{4}\norm{ \nabla^k\{\mathbf{I-P }\}f }_{L^2}^2.
\end{equation}
Hence, plugging \eqref{ME32}--\eqref{ME33} into \eqref{ME30}, we obtain %
\eqref{micro estimate 3}.
\end{proof}


\subsection{Proof of Theorems \protect\ref{theorem1}.}
In this subsection, we will combine all the energy estimates that we have derived in the previous three subsections to prove Theorem \ref{theorem1}.

We let $N\geq 3$ and then assume the a priori estimates
\begin{equation}
\norm{f(t)}_{L_{v}^{2}H_{x}^{N}}\lesssim \delta .  \label{a priori}
\end{equation}%
Let $0\leq \ell \leq N-1$. For from $k=\ell $ to $N-1$ we bound the
righthand side of \eqref{Linear ener 1} of Lemma \ref{lemma linear 1} by the
estimate \eqref{nonlinear estimate 21} of Lemma \ref{nonlinear lemma 2} and
then sum up the resulting estimates, by changing the index, we obtain
\begin{equation}
\begin{split}
& \frac{d}{dt}\sum_{\ell \leq k\leq N-1}\norm{\nabla^k f}_{L^2}^{2}+C\sum_{\ell
\leq k\leq N-1}\norm{\nabla^k\{{\bf I-P}\} f}_{\nu }^{2} \\
& \quad \leq C\delta ^{2}\left( \sum_{\ell \leq k\leq N-1}\norm{\nabla^{k+1}
f}_{L^2}^{2}+\sum_{1\leq k\leq N}\norm{\nabla^{k}\{{\bf I- P}\}f}_{\nu
}^{2}\right) .
\end{split}
\label{energy estimate 1}
\end{equation}%
For $k=N$ we bound the righthand side of \eqref{Linear ener 1} of Lemma \ref%
{lemma linear 1} by the estimates \eqref{nonlinear estimate 22} of Lemma \ref%
{nonlinear lemma 2}, and then add the resulting estimate with \eqref{energy
estimate 1}, by changing the index, to deduce that there exist constants $%
C_{2},C_{3}>0$ such that
\begin{equation}
\begin{split}
& \frac{d}{dt}\sum_{\ell \leq k\leq N}\norm{\nabla^k f}_{L^2}^{2}+C_{2}\sum_{\ell
\leq k\leq N}\norm{\nabla^k\{{\bf I-P}\} f}_{\nu }^{2} \\
& \quad \leq C_{3}\delta ^{2}\left( \sum_{\ell +1\leq k\leq N}%
\norm{\nabla^{k} f}_{L^2}^{2}+\sum_{1\leq k\leq N}\norm{\nabla^{k}\{{\bf I- P}\}f}%
_{\nu }^{2}\right) .
\end{split}
\label{energy estimate 2}
\end{equation}%
On the other hand, we sum up the estimates \eqref{linear ener 2} of Lemma %
\ref{lemma linear 2} from $k=\ell $ to $N-1$, by changing the index, to
obtain
\begin{equation}
\frac{d}{dt}\sum_{\ell \leq k\leq N-1}G_{k}+\sum_{\ell +1\leq k\leq N}\norm{
\nabla ^{k}\mathbf{P}f}_{L^2} ^{2}\leq C_{1}\sum_{\ell \leq k\leq N}\norm{
\nabla ^{k}\{\mathbf{I-P}\}f}_{L^2} ^{2}.  \label{energy estimate 3}
\end{equation}%
Then, multiplying \eqref{energy estimate 3} by a small number $\beta >0$ and
then adding the resulting inequality with \eqref{energy estimate 2}, we
obtain
\begin{equation}
\begin{split}
& \frac{d}{dt}\left( \sum_{\ell \leq k\leq N}\norm{\nabla^k f}_{L^2}^{2}+\beta
\sum_{\ell \leq k\leq N-1}G_{k}\right)  \\
& \quad +(C_{2}-C_{1}\beta )\sum_{\ell \leq k\leq N}\norm{\nabla^k\{{\bf
I-P}\} f}_{\nu }^{2}+\beta \sum_{\ell +1\leq k\leq N}\norm{ \nabla ^{k}%
\mathbf{P}f}_{L^2} ^{2} \\
& \qquad \leq C_{3}\delta ^{2}\left( \sum_{\ell +1\leq k\leq N}%
\norm{\nabla^{k} f}_{L^2}^{2}+\sum_{1\leq k\leq N}\norm{\nabla^{k}\{{\bf I- P}\}f}%
_{\nu }^{2}\right) .
\end{split}
\label{energy estimate 4}
\end{equation}

We define $\mathcal{E}_\ell(t)$ to be the expression under the time
derivative in \eqref{energy estimate 4}. We may now take $\beta$ to be
sufficiently small so that $(C_2-C_1\beta)>0$ and that $\mathcal{E}_\ell(t)$
is equivalent to $\norm{\nabla^\ell f(t)}_{L^2_vH_x^{N-\ell}}^2$ due to the
fact \eqref{G_k estimate}. On the other hand, since $\beta$ is fixed and $%
\delta$ is small, we can then absorb the first term in the right hand side
of \eqref{energy estimate 4} to have that for some constant $C_4>0$, by
adjusting the constant in the definition of $\mathcal{E}_\ell(t)$,
\begin{equation}  \label{energy estimate 5}
\frac{d}{dt}\mathcal{E}_\ell + \norm{\nabla^\ell\{{\bf I-P}\} f}_\nu^2+
\sum_{\ell+1\le k\le N}\norm{\nabla^k f}_\nu^2\le C_4 \delta^2 \sum_{1\le k
\le N} \norm{\nabla^{k} \{{\bf I- P}\}f}_\nu^2.
\end{equation}

\begin{proof}[Proof of \eqref{energy es}]
We take $\ell =0$ in \eqref{energy estimate 5}. Noticing that in this case,
we can absorb the right hand side of \eqref{energy estimate 5}, so we have,
by adjusting the constant in the definition of $\mathcal{E}_{0}(t)$,
\begin{equation}
\frac{d}{dt}\mathcal{E}_{0}+\norm{ \{{\bf I-P}\} f}_{\nu }^{2}+\sum_{1\leq
k\leq N}\norm{\nabla^k f}_{\nu }^{2}\leq 0.  \label{energy estimate 6}
\end{equation}%
Integrating \eqref{energy estimate 6} directly in time, we deduce %
\eqref{energy es}. Hence, if we assume \eqref{initial ass} for a
sufficiently small $\delta _{0}$, then a standard continuity argument closes
the a priori estimates \eqref{a priori} and thus we conclude the global
solution with the uniform bound \eqref{energy es}.
\end{proof}

Now we turn to prove \eqref{H-sbound}--\eqref{decay1}. However, we are not
able to prove them for all $s$ at this moment. We shall first prove them for
$s\in[0,1/2]$.

\begin{proof}[{Proof of \eqref{H-sbound}--\eqref{decay1} for {$s\in[0,1/2]$}}]

First, integrating in time the estimate \eqref{H-s1} of Lemma \ref{lemma H-s}%
, by the bound \eqref{energy es}, we obtain that for $s\in (0,1/2]$,
\begin{equation}
\norm{\Lambda^{-s}f(t)}_{L^2}^2 \le \norm{\Lambda^{-s}f_0}_{L^2}^2 +C\int_0^t \left(%
\norm{\{{\bf I-P}\}f (\tau)}_\nu^2+\norm{ \nabla f (\tau)}_{L^2}^2+ \norm{
\nabla^2 f(\tau)}_{L^2}^2\right)\,d\tau \le C_0.
\end{equation}
This together with \eqref{energy es} gives \eqref{H-sbound} for $s\in[0,1/2]$%
.

Next, we take $\ell =1$ in \eqref{energy estimate 5}. Noticing that in this
case we can also absorb the right hand side of \eqref{energy estimate 5}, so
we have, by adjusting the constant in the definition of $\mathcal{E}_{1}(t)$%
,
\begin{equation}
\frac{d}{dt}\mathcal{E}_{1}+\norm{\nabla \{{\bf I-P}\} f}_{\nu
}^{2}+\sum_{2\leq k\leq N}\norm{\nabla^k f}_{\nu }^{2}\leq 0.
\label{energy estimate 7}
\end{equation}%
Recalling that the energy functional $\mathcal{E}_{1}(t)$ is equivalent to $%
\norm{\nabla  f(t)}_{L_{v}^{2}H_{x}^{N-1}}^{2}$, so there is one exceptional
term $\norm{\nabla {\bf P}f(t)}_{L^2}^{2}$ that can not be bounded by the
corresponding dissipation in \eqref{energy estimate 7}. The key point is to
interpolate by  Lemma \ref{-sinte} as
\begin{equation}
\norm{\nabla {\bf P}f}_{L^2}\leq \norm{\nabla f}_{L^2}\lesssim \norm{\nabla^2 f}_{L^2}^{\frac{%
1+s}{2+s}}\norm{\Lambda^{-s} f}_{L^2}^{\frac{1}{2+s}}.  \label{inter 1}
\end{equation}%
This yields that for some $C_{5}>0$,
\begin{equation}
\norm{\nabla^2 f}_{L^2}\geq C_{5}\norm{\nabla f}_{L^2}^{1+\frac{1}{1+s}}%
\norm{\Lambda^{-s} f}_{L^2}^{-\frac{1}{1+s}}.  \label{inter 2}
\end{equation}%
Thus, by the bound \eqref{H-sbound}, then we have that there exists $C_{0}>0$
so that
\begin{equation}
\norm{\nabla^2 f}_{L^2}\geq C_{0}\norm{\nabla f}_{L^2}^{1+\frac{1}{1+s}}.
\label{inter 3}
\end{equation}%
Hence, by \eqref{inter 3} and the trivial inequality $\norm{\cdot}\leq %
\norm{\cdot}_{\nu }$, we deduce from \eqref{energy estimate 7} that
\begin{equation}
\frac{d}{dt}\mathcal{E}_{1}+C_{0}(\mathcal{E}_{1})^{1+\frac{1}{1+s}}\leq 0.
\label{energy estimate 8}
\end{equation}%
Solving this inequality directly and by \eqref{energy es} again, we obtain
that
\begin{equation}
\mathcal{E}_{1}(t)\leq C_{0}(1+t)^{-(1+s)}.  \label{energy inequality 1}
\end{equation}%
This gives \eqref{decay1} for $\ell =1$. While for $-s<\ell <1$, %
\eqref{decay1} follows by the interpolation.
\end{proof}

\begin{proof}[Proof of \eqref{H-sbound}--\eqref{decay1} for {$s\in(1/2,3/2)$}%
]
Notice that the arguments for the case $s\in[0,1/2]$ can not be applied to
this case. However, observing that we have $f_0\in L^2_v\dot{H}_x^{-1/2}$
since $L^2_v\dot{H}_x^{-s}\cap L^2_vL^2_x\subset L^2_v\dot{H}_x^{-s^{\prime}}
$ for any $s^{\prime}\in [0,s]$, we then deduce from what we have proved for %
\eqref{H-sbound} and \eqref{decay1} with $s=1/2$ that the following decay
result holds:
\begin{equation}  \label{decay11}
\sum_{\ell\le k\le N}\norm{\nabla^k f(t)}_{L^2}^2\le C_0(1+t)^{-(\ell+ \frac{1}{2}%
)}\, \hbox{ for }-\frac{1}{2}\le\ell\le 1.
\end{equation}
Hence, by \eqref{decay11} and \eqref{energy es}, we deduce from \eqref{H-s2}
that for $s\in (1/2,3/2)$,
\begin{equation}  \label{sssbound}
\begin{split}
\norm{\Lambda^{-s}f(t)}_{L^2}^2&\le \norm{\Lambda^{-s}f_0}_{L^2}^2+C \int_0^t\left(%
\norm{\{{\bf I-P}\}f (\tau)}_\nu^2+\norm{ f(\tau) }_{L^2}^{2s+1}\norm{ \nabla f
(\tau) }_{L^2}^{3-2s}\right)\,d\tau. \\
&\le C_0+C_0\int_0^t(1+\tau)^{-(5/2-s)}\,d\tau\le C_0.
\end{split}%
\end{equation}
This proves \eqref{H-sbound} for $s\in (1/2,3/2)$. Now that we have proved %
\eqref{sssbound}, we may repeat the arguments leading to \eqref{decay1} for $%
s\in [0,1/2]$ to obtain \eqref{decay1} for $s\in (1/2,3/2)$.
\end{proof}

\begin{proof}[Proof of \eqref{decay1'}]
Applying the Gronwall inequality to \eqref{micro estimate 1}, by %
\eqref{decay1} with $\ell=1$, we obtain
\begin{equation}  \label{micro inequality}
\begin{split}
\norm{\{\mathbf{I-P }\}f(t) }_{L^2}^2 &\le e^{-\sigma_0t}\norm{\{\mathbf{I-P }%
\}f_0}_{L^2}^2+C\int_0^t e^{-\sigma_0(t-\tau)}\norm{\nabla f(\tau)}_{L^2}^2\,d\tau \\
&\le e^{-\sigma_0t}\norm{f_0}_{L^2}^2+C_0\int_0^t
e^{-\sigma_0(t-\tau)}(1+\tau)^{-(1+s)}\,d\tau \\
&\le C_0 (1+t)^{-(1+s)}.
\end{split}%
\end{equation}
This gives \eqref{decay1'}.
\end{proof}

\begin{proof}[Proof of \eqref{decay2}]
Due to the estimates \eqref{energy es} and \eqref{micro estimate 2}, if we
assume \eqref{initial ass2} for a sufficiently small $\delta _{0}>0$, then
we have the estimates
\begin{equation}
\norm{f(t)}_{L_{v}^{2}H_{x}^{N}}+\norm{f(t)}_{\nu }\lesssim \delta .
\label{a priori es}
\end{equation}%
This allows us to use Lemma \ref{nonlinear lemma 3} in place of Lemma \ref%
{nonlinear lemma 2} to improve the estimate \eqref{energy estimate 5} as
\begin{equation}
\frac{d}{dt}\mathcal{E}_{\ell }+\norm{\nabla^\ell\{{\bf I-P}\}f}_{\nu
}^{2}+\sum_{\ell +1\leq k\leq N}\norm{\nabla^k f}_{\nu }^{2}\leq 0\,\hbox{
for }\ell =0,\dots ,N-1.  \label{energy estimate 9}
\end{equation}%
Then we interpolate by using Lemma \ref{-sinte} as,
\begin{equation}
\norm{\nabla^\ell f}_{L^2}\lesssim \norm{\nabla^{\ell+1}f}_{L^2}^{\frac{\ell +s}{\ell
+s+1}}\norm{\Lambda^{-s} f}_{L^2}^{\frac{1}{\ell +s+1}}.  \label{inter 11}
\end{equation}%
This together with \eqref{H-sbound} yields that there exists $C_{0}>0$ such
that for $-s<\ell \leq N-1$,
\begin{equation}
\norm{\nabla^{\ell+1} f}_{L^2}\geq C\left( \norm{\nabla^\ell f}_{L^2}^{2}\right) ^{1+%
\frac{1}{\ell +s}}\norm{\Lambda^{-s} f}_{L^2}^{-\frac{1}{\ell +s}}\geq C_{0}\left( %
\norm{\nabla^\ell f}_{L^2}^{2}\right) ^{1+\frac{1}{\ell +s}}.  \label{inter 21}
\end{equation}%
Hence, by \eqref{inter 21} and the trivial inequality $\norm{\cdot}_{L^2}\leq %
\norm{\cdot}_{\nu }$, we deduce from \eqref{energy estimate 9} that
\begin{equation}
\frac{d}{dt}\mathcal{E}_{\ell }+C_{0}(\mathcal{E}_{\ell })^{1+\frac{1}{\ell
+s}}\leq 1\,\hbox{ for }\ell =1,\dots ,N-1.  \label{energy estimate 18}
\end{equation}%
Solving this inequality directly and by \eqref{energy es} again, we obtain
that
\begin{equation}
\mathcal{E}_{\ell }(t)\leq C_{0}(1+t)^{-(\ell +s)}\,\hbox{ for }\ell
=1,\dots ,N-1.  \label{energy inequality 1''}
\end{equation}%
Thus, by \eqref{H-sbound}, \eqref{energy inequality 1''} and the
interpolation, we deduce \eqref{decay2}.
\end{proof}

\begin{proof}[Proof of \eqref{decay2'}]
The estimates \eqref{a priori es} allows us to have the estimates %
\eqref{micro estimate 3} of Lemma \ref{lemma micro 3}. Hence, applying the
Gronwall inequality to \eqref{micro estimate 3} with $k=1,\cdots ,N-2$, by %
\eqref{decay2} with $\ell =k+1$, we obtain
\begin{equation}
\begin{split}
\norm{ \nabla ^{k}\{\mathbf{I-P}\}f(t)}_{L^2} ^{2}& \leq e^{-\sigma _{0}t}\norm{
\nabla ^{k}\{\mathbf{I-P}\}f_{0}}_{L^2} ^{2}+C\int_{0}^{t}e^{-\sigma
_{0}(t-\tau )}\norm{\nabla^{k+1} f(\tau)}_{L^2}^{2}\,d\tau  \\
& \leq e^{-\sigma _{0}t}\norm{\nabla ^{k}f_{0}}_{L^2}
^{2}+C_{0}\int_{0}^{t}e^{-\sigma _{0}(t-\tau )}(1+\tau )^{-(k+1+s)}\,d\tau
\\
& \leq C_{0}(1+t)^{-(k+1+s)}.
\end{split}
\label{micro inequality 3}
\end{equation}%
This proves \eqref{decay2'} for $\ell =1,\dots ,N-2$. Which for $-s<\ell <N-2
$, \eqref{decay2'} follows by the interpolation. The proof of Theorem \ref%
{theorem1} is completed.
\end{proof}

\appendix


\section{Analytic tools}

\label{section_appendix} 


\subsection{Sobolev type inequalities}


We will extensively use the Sobolev interpolation of the Gagliardo-Nirenberg
inequality.

\begin{lemma}
\label{1interpolation}  Let $0\le m, \alpha\le \ell$, then we have
\begin{equation}
\norm{\nabla^\alpha f}_{L^p}\lesssim \norm{ \nabla^mf}_{L^2}^{1-\theta}%
\norm{ \nabla^\ell f}_{L^2}^{\theta}
\end{equation}
where $0\le \theta\le 1$ and $\alpha$ satisfy
\begin{equation}
\frac{1}{p}-\frac{\alpha}{3}=\left(\frac{1}{2}-\frac{m}{3}%
\right)(1-\theta)+\left(\frac{1}{2}-\frac{\ell}{3}\right)\theta.
\end{equation}

\begin{proof}
This can be found in \cite[pp. 125, THEOREM]{N1959}.
\end{proof}
\end{lemma}

For the Boltzmann equation, we shall use the corresponding Sobolev
interpolation of the Gagliardo-Nirenberg inequality for the functions on $%
\r3_x\times\r3_v$.

\begin{lemma}
\label{interpolation} Let $0\le m, \alpha\le \ell$. Let $w(v)$ be any weight
function of $v$, then we have
\begin{equation}  \label{GN00}
\left(\int_{\mathbb{R}^3_v}w\norm{\nabla^\alpha f}_{L^p_x}^2\,dv\right)^{%
\frac{1}{2}} \lesssim \left(\int_{\mathbb{R}^3_v}w\norm{ \nabla^mf}%
_{L^2_x}^2\,dv\right)^{\frac{1-\theta}{2}}\left(\int_{\mathbb{R}^3_v}w\norm{
\nabla^\ell f}_{L^2_x}^2\,dv\right)^{\frac{\theta}{2}}
\end{equation}
where $0\le \theta\le 1$ and $\alpha$ satisfy
\begin{equation}
\frac{1}{p}-\frac{\alpha}{3}=\left(\frac{1}{2}-\frac{m}{3}%
\right)(1-\theta)+\left(\frac{1}{2}-\frac{\ell}{3}\right)\theta.
\end{equation}
\end{lemma}

\begin{proof}
For any function $f(x,v)$, by Lemma \ref{1interpolation}, we have
\begin{equation}
\norm{\nabla^\alpha f}_{L_{x}^{p}}\lesssim \norm{ \nabla^mf}%
_{L_{x}^{2}}^{1-\theta }\norm{ \nabla^\ell f}_{L_{x}^{2}}^{\theta }.
\label{GN1}
\end{equation}%
Taking the square of \eqref{GN1} and then multiplying by $w(v)$, integrating
over $\mathbb{R}_{v}^{3}$, by H\"{o}lder's inequality, we obtain
\begin{equation}
\begin{split}
\int_{\mathbb{R}_{v}^{3}}w\norm{\nabla^\alpha f}_{L_{x}^{p}}^{2}dv& \lesssim
\int_{\mathbb{R}_{v}^{3}}w\norm{ \nabla^mf}_{L_{x}^{2}}^{2(1-\theta )}\norm{
\nabla^\ell f}_{L_{x}^{2}}^{2\theta }\,dv \\
& =\int_{\mathbb{R}_{v}^{3}}\left( w^{\frac{1}{2}}\norm{\nabla^mf}%
_{L_{x}^{2}}\right) ^{2(1-\theta )}\left( w^{\frac{1}{2}}\norm{\nabla^\ell f}%
_{L_{x}^{2}}\right) ^{2\theta }\,dv \\
& \leq \left( \int_{\mathbb{R}_{v}^{3}}\left( w^{\frac{1}{2}}\norm{\nabla^mf}%
_{L_{x}^{2}}\right) ^{2}\,dv\right) ^{1-\theta }\left( \int_{\mathbb{R}%
_{v}^{3}}\left( w^{\frac{1}{2}}\norm{ \nabla^\ell f}_{L_{x}^{2}}\right)
^{2\theta }\,dv\right) ^{\theta }.
\end{split}
\label{GN2}
\end{equation}%
Taking the square root of \eqref{GN2}, we deduce \eqref{GN00}.
\end{proof}

We recall the following commutator estimate:

\begin{lemma}
\label{1commutator} Let $m\ge 1$ be an integer and define the commutator
\begin{equation}  \label{1commuta}
[\nabla^m,f]g=\nabla^m(fg)-f\nabla^mg.
\end{equation}
Then we have
\begin{equation}
\norm{[\nabla^m,f]g}_{L^2} \lesssim \norm{\nabla f}_{L^\infty}%
\norm{\nabla^{m-1}g}_{L^2}+\norm{\nabla^m f}_{L^2}\norm{ g}_{L^\infty}.
\end{equation}
\end{lemma}

\begin{proof}
It can be proved by using Lemma \ref{1interpolation}, see \cite[pp. 98,
Lemma 3.4]{MB} for instance.
\end{proof}


\subsection{Negative Sobolev norms}


We define the operator $\Lambda^s, s\in \mathbb{R}$ by
\begin{equation}  \label{1Lambdas}
\Lambda^s f(x)=\int_{\mathbb{R}^3}|\xi|^s\hat{f}(\xi)e^{2\pi
ix\cdot\xi}\,d\xi,
\end{equation}
where $\hat{f}$ is the Fourier transform of $f$. We define the homogeneous
Sobolev space $\dot{H}^s$ of all $f$ for which $\norm{f}_{\dot{H}^s}$ is
finite, where
\begin{equation}  \label{1snorm}
\norm{f}_{\dot{H}^s}:=\norm{\Lambda^s f}_{L^2}=\norm{|\xi|^s \hat{f}}_{L^2}.
\end{equation}
We will use the non-positive index $s$. For convenience, we will change the
index to be ``$-s$" with $s\ge 0$. We will employ the following special
Sobolev interpolation:

\begin{lemma}
\label{1-sinte} Let $s\ge 0$ and $\ell\ge 0$, then we have
\begin{equation}  \label{1-sinterpolation}
\norm{\nabla^\ell f}_{L^2}\le \norm{\nabla^{\ell+1} f}_{L^2}^{1-\theta}%
\norm{\Lambda^{-s}f}_{L^2}^\theta, \hbox{ where }\theta=\frac{1}{\ell+1+s}.
\end{equation}
\end{lemma}

\begin{proof}
By the Parseval theorem, the definition of \eqref{1snorm} and H\"older's
inequality, we have
\begin{equation}
\norm{\nabla^\ell f}_{L^2} =\norm{|\xi|^\ell \hat{f}}_{L^2}\le %
\norm{|\xi|^{\ell+1} \hat{f}}_{L^2}^{1-\theta}\norm{|\xi|^{-s} \hat{f}}%
_{L^2}^\theta =\norm{\nabla^{\ell+1}f}_{L^2}^{1-\theta}\norm{ f}_{\Dot{H}%
^{-s}}^\theta.
\end{equation}
\end{proof}

For the Boltzmann equation, we shall use the corresponding Sobolev
interpolation for the functions on $\r3_x\times\r3_v$.

\begin{lemma}
\label{-sinte} Let $s\ge 0$ and $\ell\ge 0$, then we have
\begin{equation}
\norm{\nabla^\ell f}_{L^2}\lesssim \norm{\nabla^{\ell+1} f}_{L^2}^{1-\theta}%
\norm{\Lambda^{-s}f}_{L^2}^\theta, \hbox{ where }\theta=\frac{1}{\ell+1+s}.
\end{equation}
\end{lemma}

\begin{proof}
It follows by further taking the $L^2$ norm of \eqref{1-sinterpolation} over
$\mathbb{R}_v^3$.
\end{proof}

If $s\in(0,3)$, $\Lambda^{-s}f$ defined by \eqref{1Lambdas} is the Riesz
potential. The Hardy-Littlewood-Sobolev theorem implies the following $L^p$
inequality for the Riesz potential:

\begin{lemma}
\label{1Riesz} Let $0<s<3,\ 1<p<q<\infty,\ 1/q+s/3=1/p$, then
\begin{equation}  \label{1Riesz es}
\norm{\Lambda^{-s}f}_{L^q}\lesssim\norm{ f}_{L^p}.
\end{equation}
\end{lemma}

\begin{proof}
See \cite[pp. 119, Theorem 1]{S}.
\end{proof}


\subsection{Minkowski's inequality}


In estimating the nonlinear terms for the Boltzmann equation, it is crucial
to use the Minkowski's integral inequality to exchange the orders of
integration over $x$ and $v$.

\begin{lemma}
\label{Minkowski} Let $1\le p<\infty$. Let $f$ be a measurable function on $%
\mathbb{R}_y^3\times \mathbb{R}_z^3$, then we have
\begin{equation}  \label{mink00}
\left(\int_{\mathbb{R}_z^3}\left(\int_{\mathbb{R}_y^3}|f(y,z)|\,dy\right)^p%
\,dz\right)^\frac{1}{p} \le\int_{\mathbb{R}_y^3}\left(\int_{\mathbb{R}%
_z^3}|f(y,z)|^p\,dz\right)^\frac{1}{p}\,dy.
\end{equation}
In particular, for $1\le p\le q\le \infty$, we have
\begin{equation}  \label{min es}
\norm{f}_{L^q_zL^p_y}\le \norm{f}_{L^p_yL^q_z}.
\end{equation}
\end{lemma}

\begin{proof}
The inequality \eqref{mink00} can be found in \cite[pp. 271, A.1]{S}, hence
it remains to prove \eqref{min es}. For $q=\infty$, we have
\begin{equation}
\norm{f}_{L^\infty_zL^p_y}=\sup_{z\in
\r3}\left(\int_{\r3_y}|f|^p\,dy\right)^{1/p}\le
\left(\int_{\r3_y}\left(\sup_{z\in \r3}|f|\right)^p\,dv\right)^{1/p}=\norm{f}%
_{L^p_yL^\infty_z}.
\end{equation}
For $q<\infty$ and hence $1\le q/p<\infty$, then by \eqref{mink00}, we have
\begin{equation}
\norm{f}_{L^q_zL^p_y}=
\left(\int_{\r3_z}\left(\int_{\r3_y}|f|^p\,dy\right)^{q/p}\,dz\right)^{1/q}%
\le
\left(\int_{\r3_y}\left(\int_{\r3_z}|f|^q\,dz\right)^{p/q}\,dv\right)^{1/p}=%
\norm{f}_{L^p_yL^q_z}.
\end{equation}
\end{proof}


\subsection{Boltzmann collision operators}


Now, we collect some useful estimates of the linear collision operator.

\begin{lemma}
\label{linearcol1} $\langle Lh_1,h_2\rangle=\langle h_1,Lh_2\rangle$, $%
\langle Lh,h\rangle\ge0$. And $L h=0$ if and only if $h=\mathbf{P} h$.
Moreover, there exist a constant $\sigma_0>0$ such that
\begin{equation}  \label{lineares1}
\langle L h,h\rangle \ge \sigma_0\left \vert\{\mathbf{I}-\mathbf{P}\}h\right
\vert_\nu^2,
\end{equation}
and
\begin{equation}  \label{lineares2}
\langle\nu L h, h\rangle \ge \frac{1}{2}|\nu h|_2^2-C|h|_\nu^2.
\end{equation}
\end{lemma}

\begin{proof}
We refer to \cite[Lemma 3.2]{G2006} for \eqref{lineares1}, and \cite[Lemma
3.3]{G2006} for \eqref{lineares2}.
\end{proof}

Next, we collect some useful estimates of the nonlinear collision operator.

\begin{lemma}
\label{nonlinearcol1} There exists $C>0$ such that
\begin{equation}  \label{nonlineares2}
|\langle\Gamma(h_1,h_2),h_3\rangle|+|\langle\Gamma(h_2,h_1), h_3\rangle|\le
C\sup_{v}\{\nu^3 h_3\}|h_1|_2|h_2|_2.
\end{equation}
Moreover, for any $0\le \eta \le 1$, we have
\begin{equation}  \label{nonlineares3}
|\nu^{-\eta}\Gamma(h_1,h_2)|_2\le C\left\{|\nu^{1-\eta}
h_1|_2|h_2|_2+|\nu^{1-\eta}h_2|_2|h_1|_2\right\}.
\end{equation}
\end{lemma}

\begin{proof}
We refer to \cite[Lemma 2.3]{G2002} for \eqref{nonlineares2}, and \cite[%
Lemma 2.7]{UY2006} for \eqref{nonlineares3}.
\end{proof}

\

\end{document}